\documentclass[11pt, a4paper, oneside]{amsart}

\usepackage[english]{babel}
\usepackage{amsmath, amsthm, amsfonts, mathrsfs, amssymb}
\usepackage{mathtools}
\mathtoolsset{centercolon}
\usepackage{booktabs}
\usepackage[shortlabels]{enumitem}
\setlist[itemize]{leftmargin=20pt}
\usepackage[colorlinks, citecolor = blue]{hyperref}
\usepackage[hmargin=3cm,vmargin=3cm]{geometry}
\usepackage[color=green!40]{todonotes}
\usepackage{comment}

\usepackage{color}
\usepackage{graphicx}
\usepackage{tikz-cd}
\usetikzlibrary{arrows}

\newcommand{\Z}{\ensuremath{\mathbf{Z}}}

\newcommand{\R}{\ensuremath{\mathbf{R}}}

\newcommand{\mc}{\mathcal}

\DeclareMathOperator{\ind}{\mathbf{1}}

\DeclareMathOperator*{\esssup}{ess\,sup}

\renewcommand{\emptyset}{\varnothing}
\def\avint_#1{\mathchoice{\mathop{\kern 0.2em\vrule width 0.6em height 0.69678ex depth -0.58065ex \kern -0.8em \intop}\nolimits_{\kern -0.4em#1}}{\mathop{\kern 0.1em\vrule width 0.5em height 0.69678ex depth -0.60387ex \kern -0.6em \intop}\nolimits_{#1}} {\mathop{\kern 0.1em\vrule width 0.5em height 0.69678ex depth -0.60387ex \kern -0.6em \intop}\nolimits_{#1}} {\mathop{\kern 0.1em\vrule width 0.5em height 0.69678ex depth -0.60387ex \kern -0.6em \intop}\nolimits_{#1}}}

\newcommand{\avg}[1]{\langle #1 \rangle}

\newtheorem{theorem}{Theorem}

\newtheorem{lemma}[theorem]{Lemma}
\newtheorem{proposition}[theorem]{Proposition}

\newtheorem{TheoremLetter}{Theorem}
{}

\theoremstyle{remark}
\newtheorem{remark}[theorem]{Remark}

\theoremstyle{definition}

\numberwithin{theorem}{section}
\numberwithin{equation}{section}
\title{Weighted weak-type bounds for multilinear singular integrals}

\author{Zoe Nieraeth}
\address{Zoe Nieraeth (she/her), BCAM\textendash Basque Center for Applied Mathematics, Bilbao, Spain}
\email{zoe.nieraeth@gmail.com}
\thanks{Z. N. is supported by the grant Juan de la Cierva formación 2021 FJC2021-046837-I, the Basque Government through the BERC 2022-2025 program, by the Spanish State Research Agency project PID2020-113156GB-I00/AEI/10.13039/501100011033 and through BCAM Severo Ochoa excellence accreditation SEV-2023-2026.}

\author{Cody B. Stockdale}
\address{Cody B. Stockdale (he/him), School of Mathematical Sciences and Statistics, Clemson University, Clemson, SC 29634, USA}
\email{cbstock@clemson.edu}

\author{Brandon Sweeting}
\address{Brandon Sweeting (he/him), Department of Mathematics, University of Alabama, Tuscaloosa, AL 35487, USA}
\email{bssweeting@ua.edu}

\allowdisplaybreaks

\begin{document}




\begin{abstract}
We establish analogs of sharp weighted weak-type bounds for $m$-sublinear operators satisfying sparse form domination, including multilinear Calder\'on-Zygmund singular integrals. Our results, which hold for general $\vec{p} \in [1,\infty)^m$ and feature quantitative improvements, rely on new local testing conditions and good-$\lambda$ inequalities. We address weak-type bounds in both the change of measure and multiplier settings.
\end{abstract}

\maketitle

\section{Introduction}

The weighted strong-type bound 
\begin{align}\label{LinearStrongType}
    \|Tf\|_{L^p(\R^d,w)}\lesssim_w\|f\|_{L^{p}(\R^d,w)}
\end{align}
implies two particularly interesting weak-type bounds, 
namely, the usual weak-type formulation where the weight is treated as a measure: 
\begin{align}\label{LinearWeakType}
    \|Tf\|_{L^{p,\infty}(\R^d,w)}\lesssim_w \|f\|_{L^{p}(\R^d,w)},
\end{align}
and the multiplier weak-type bound:
\begin{align}\label{LinearMultiplierWeakType}   
    \|T(fw^{-1/p})w^{1/p}\|_{L^{p,\infty}(\R^d)}\lesssim_w\|f\|_{L^{p}(\R^d)}.
\end{align}
It is well-known that if $T$ is a Calder\'on-Zygmund operator, $p \in (1,\infty)$, and a weight $w$ satisfies Muckenhoupt's $A_p$ condition 
$$
    [w]_{A_p}:=\sup_{Q} \langle w\rangle_{1,Q}\langle w^{1-p'}\rangle_{1,Q}^{p-1}<\infty,
$$
then \eqref{LinearStrongType} holds for all $f \in L^p(\R^d,w)$ and, hence, so do \eqref{LinearWeakType} and \eqref{LinearMultiplierWeakType}. Moreover, if 
$$
    [w]_{A_1}:=\sup_{Q}\langle w\rangle_{1,Q} \langle w^{-1}\rangle_{\infty,Q}
    <\infty,
$$
then \eqref{LinearWeakType} and \eqref{LinearMultiplierWeakType} both hold with $p=1$. 

Quantitative versions of these inequalities are much more intricate. In \cite{Hy12}, Hyt\"onen famously proved that if $T$ is a Calder\'on-Zygmund operator, $p\in(1,\infty)$, and $w \in A_p$, then
$$
    \|Tf\|_{L^p(\R^d,w)}\lesssim [w]_{A_p}^{\frac{1}{p} \max\left(p,p'\right)}\|f\|_{L^p(\R^d,w)}.
$$
Improvements can be made to the weak-type bounds inherited from the above sharp strong-type bound. Indeed, if $T$ is a Calder\'on-Zygmund operator, then 
$$
    \|Tf\|_{L^{p,\infty}(\R^d,w)}\lesssim [w]_{A_p}\|f\|_{L^p(\R^d,w)}
$$
for $p>1$ and
$$
    \|Tf\|_{L^{1,\infty}(\R^d,w)}\lesssim (1+\log[w]_{A_1})[w]_{A_1}\|f\|_{L^1(\R^d,w)}, 
$$
and in the multiplier setting

$$
    \|T(fw^{-1/p})w^{1/p}\|_{L^{p,\infty}(\R^d)} \lesssim [w]_{A_p}^{1+1/p}\|f\|_{L^p(\R^d)}
$$
for $p\ge 1$, see \cite{HLMORSU12, LOP09, CS23}.
The dependence in the first inequality above is optimal with respect to $[w]_{A_p}$, and in the case $p=1$, the latter two 
constants above are also sharp, see \cite{HLMORSU12, LNO20,LLORr23}. 

These bounds are further improved by introducing the smaller Fujii-Wilson constant 
\[
[w]_{\text{FW}}:=\sup_{Q}\frac{1}{w(Q)}\int_Q\!M(w\ind_Q)\,\mathrm{d}x \lesssim [w]_{A_p}
\]
as one has the dependence of $[w]_{\text{FW}}^{1/p'}[w]_{A_p}^{1/p}$, $(1+\log[w]_{\text{FW}})[w]_{A_1}$, and $[w]_{\text{FW}}[w]_{A_p}^{1/p}$ in the above three inequalities, respectively, see \cite{HLMORSU12, HP13, CS23}. Moreover, all of these bounds hold for the more general class of operators $T$ satisfying sparse form domination, which means that for all $f,g \in L_0^{\infty}(\R^d)$ there exists a sparse collection $\mathcal{S}$ such that 
$$
\int_{\R^d}\!|Tf||g|\,\mathrm{d}x\lesssim \sum_{Q\in\mathcal{S}}\langle f\rangle_{1,Q}\langle g\rangle_{1,Q}|Q|,
$$
see \cite{Mo12,HLMORSU12,FN19,CS23}. While it is true that Calder\'on-Zygmund operators satisfy the stronger pointwise sparse domination
\[
|Tf(x)|\lesssim\sum_{Q\in\mc{S}}\langle f\rangle_{1,Q}\ind_Q(x),
\]
the class of operators satisfying sparse form domination is even larger and includes some non-integral operators, see \cite{BFP16} and the references therein.

The classical Calder\'on-Zygmund theory was extended to the multilinear setting by Grafakos and Torres in their seminal paper \cite{GT02a}. Connecting the weighted and the multilinear settings, Lerner, Ombrosi, P\'erez, Torres, and Trujillo-Gonz\'alez introduced the multilinear $A_{\vec{p}}$ classes and characterized the weighted bounds for multilinear Calder\'on-Zygmund operators in \cite{LOPTT09}. 
While the results in \cite{LOPTT09} are qualitatively sharp, quantitative bounds in terms of $A_{\vec{p}}$ characteristics are less well understood; progress in this direction was made using sparse domination in \cite{LMS14, CR16, LN15, Zor19, Zh23}. 

\vspace{.05in}
\emph{We establish new and improved weighted weak-type estimates for $m$-sublinear operators satisfying sparse form domination in the change of measure and the multiplier settings. This framework applies to multilinear Calder\'on-Zygmund operators with Dini continuous kernels and their maximal truncations, and to multilinear multipliers that are invariant under simultaneous modulations of the input functions, see \cite{CR16, LN15, DHL18, CDO18}. The novelty of our results includes extensions to general $\vec{p} \in [1,\infty)^m$, quantitative improvements, and the consideration of multilinear multiplier weak-type bounds. Our arguments involve new local testing conditions and good-$\lambda$ inequalities.}
\vspace{.05in}

We say that an $m$-sublinear operator satisfies sparse form domination if for every $f_1,\ldots, f_m,g \in L_0^{\infty}(\R^d)$, there exists a sparse collection $\mathcal{S}$ such that
\begin{equation}\label{eq:introsparseform}
    \int_{\R^d}\!|T\vec{f}||g|\,\mathrm{d}x\lesssim \sum_{Q\in\mathcal{S}}\Big(\prod_{j=1}^{m}\langle f_j\rangle_{1,Q}\Big)\langle g\rangle_{1,Q}|Q|.
\end{equation}
For $p>0$ and a weight $w$, we define
\[
\|f\|_{L^p_w(\R^d)}:=\|fw\|_{L^p(\R^d)}.
\]
For exponents $\vec{p}=(p_1,\ldots,p_m)$ and weights $\vec{w}=(w_1,\ldots,w_m)$, we write
\[
L^{\vec{p}}_{\vec{w}}(\R^d):=\prod_{j=1}^mL^{p_j}_{w_j}(\R^d)
\quad \text{and}\quad \|\vec{f}\|_{L^{\vec{p}}_{\vec{w}}(\R^d)}:=\prod_{j=1}^m\|f_j\|_{L^{p_j}_{w_j}(\R^d)}.
\] 
We extend this notation to weak-Lebesgue spaces by defining
\[
\|f\|_{L^{p,\infty}_w(\R^d)}:=\sup_{\lambda>0}\|\lambda\ind_{\{|f|>\lambda\}}\|_{L^p_w(\R^d)}=\sup_{\lambda>0}\lambda w^p\big(\{|f|>\lambda\}\big)^{\frac{1}{p}}.
\]
When we wish to treat a weight $v$ as a measure, we write
\[
L^p(\R^d,v)\quad \text{and}\quad L^{p,\infty}(\R^d,v),
\]
which, for finite $p$, respectively coincide with $L^p_w(\R^d)$ and $L^{p,\infty}_w(\R^d)$ when $v=w^p$. See \cite{LN23b} for further discussion on this perspective on weights as multipliers and as measures. For $\vec{p} \in [1,\infty]^m$, the multilinear Muckenhoupt condition $\vec{w} \in A_{\vec{p}}$ takes the form 
$$
    [\vec{w}]_{\vec{p}}:= \sup_{Q} \langle w\rangle_{p,Q}\prod_{j=1}^m \langle w_j^{-1}\rangle_{p_j',Q}<\infty,
$$
where $w:=\prod_{j=1}^m w_j$ and $p\in[\tfrac{1}{m},\infty)$ satisfies $\frac{1}{p}=\sum_{j=1}^m\frac{1}{p_j}$. 
In the case $m=1$, we have that $\vec{w} \in A_{\vec{p}}$ if and only if $w^p \in A_p$, and $[\vec{w}]_{\vec{p}}=[w^p]_{A_p}^{1/p}$.

Our inequalities take the following forms: 
\begin{align}\label{MultilinearWeakType}
    \|T\vec{f}\|_{L^{p,\infty}_w(\R^d)}\lesssim_{\vec{w}} \|\vec{f}\|_{L^{\vec{p}}_{\vec{w}}(\R^d)}
\end{align}
and 
\begin{align}\label{MultilinearMultiplierWeakType}
    \|T(\vec{f}/{\vec{w}})w\|_{L^{p,\infty}(\R^d)}\lesssim_{\vec{w}} \|\vec{f}\|_{L^{\vec{p}}(\R^d)}.
\end{align}
Note that \eqref{MultilinearWeakType} and \eqref{MultilinearMultiplierWeakType} respectively generalize \eqref{LinearWeakType} and \eqref{LinearMultiplierWeakType} to the multilinear setting. 

\begin{TheoremLetter}\label{thm:E}
Let $T$ be an $m$-sublinear operator satisfying sparse form domination, let $\vec{p}\in [1,\infty]^m$, and let $p \in [\tfrac{1}{m},\infty)$ satisfy $\tfrac{1}{p}=\sum_{j=1}^m\tfrac{1}{p_j}$. If $\vec{w}\in A_{\vec{p}}$, then
\[
\|T\vec{f}\|_{L^{p,\infty}_w(\R^d)}\lesssim [w^p]_{\text{FW}}[\vec{w}]_{\vec{p}}\|\vec{f}\|_{L^{\vec{p}}_{\vec{w}}(\R^d)}
\]
for all $\vec{f} \in L^{\vec{p}}_{\vec{w}}(\R^d)$. Moreover, if $\vec{p} \in (1,\infty)^m$ with $p >1$, then 
\[
\|T\vec{f}\|_{L^{p,\infty}_w(\R^d)}\lesssim C_{\vec{w}}[\vec{w}]_{\vec{p}}\|\vec{f}\|_{L^{\vec{p}}_{\vec{w}}(\R^d)}
\]
for all $\vec{f} \in L^{\vec{p}}_{\vec{w}}(\R^d)$, where 
\[
C_{\vec{w}}:= \max_{k\in\{1,\ldots,m\}}
\min\Big([w_1^{-p_1'}]_{\text{FW}},\ldots,[w_{k-1}^{-p_{k-1}'}]_{\text{FW}},[w^p]_{\text{FW}},[w_{k+1}^{-p_{k+1}'}]_{\text{FW}},\ldots,[w_m^{-p_m'}]_{\text{FW}}\Big)^{\frac{1}{p_k'}}.
\]
\end{TheoremLetter}

The second part of Theorem \ref{thm:E} implies that if $\vec{p} \in (1,\infty)^m$ with $p>1$ and $\vec{w} \in A_{\vec{p}}$, then 
\begin{equation*}
\|T\vec{f}\|_{L^{p,\infty}_w(\R^d)}\lesssim [\vec{w}]^{\min(\alpha,\beta)}_{\vec{p}}\|\vec{f}\|_{L^{\vec{p}}_{\vec{w}}(\R^d)}
\end{equation*}
for all $\vec{f} \in L^{\vec{p}}_{\vec{w}}(\R^d)$, where
\[
\alpha:=1+\max_{k\in\{1,\ldots,m\}}\min\Big(\tfrac{p_1'}{p_k'},\ldots,\tfrac{p_{k-1}'}{p_k'},\tfrac{p}{p_k'},\tfrac{p_{k+1}'}{p_k'},\ldots,\tfrac{p_m'}{p_k'}\Big)
\]
and
\[
\beta:=\max(p,p_1',\ldots,p_m').
\]
The exponent $\beta$ above comes from the following sharp strong-type estimate of \cite{LMS14, CR16, LN15}: if $\vec{p} \in (1,\infty]^m$ with $p \in (\tfrac{1}{m},\infty)$ and $\vec{w} \in A_{\vec{p}}$, then 
\begin{equation}\label{MultilinearSharpStrongType}
    \|T\vec{f}\|_{L^p_w(\R^d)} \lesssim [\vec{w}]_{\vec{p}}^{\max (p,p_1',\ldots,p_m')}\|\vec{f}\|_{L_{\vec{w}}^{\vec{p}}(\R^d)}.
\end{equation}
for all $\vec{f} \in L_{\vec{w}}^{\vec{p}}(\R^d)$. Note that $\alpha < \beta$ for certain $\vec{p}$ and, hence, Theorem~\ref{thm:E} improves the bound inherited \eqref{MultilinearSharpStrongType} for such $\vec{p}$, see \cite{Zh23} for the case $m=2$. Further, Theorem~\ref{thm:E} gives a dependence of $[\vec{w}]_{\vec{p}}^{p+1}$ for general $\vec{p} \in [1,\infty]^m$, which provides an improvement over \eqref{MultilinearSharpStrongType} for all $\vec{p}$ for which $\tfrac{1}{m}\leq p\leq \tfrac{1}{(m+\frac{1}{4})^{\frac{1}{2}}-\frac{1}{2}}$. Indeed, in this case
\[
1-(p+1)(1-\tfrac{1}{pm})=\tfrac{1}{pm}-p+\tfrac{1}{m}=\tfrac{p}{m}\big(\big(\tfrac{1}{p}+\tfrac{1}{2})^2-(m+\tfrac{1}{4}\big)\big)\geq 0,
\]
so that
\[
\max(p,p_1',\ldots,p_m')=\frac{1}{1-\max(\frac{1}{p_1},\ldots,\frac{1}{p_m})}\geq\frac{1}{1-\frac{1}{pm}}\geq p+1.
\]
Note that when $m=2$, this is the entire range $\tfrac{1}{2}\leq p\leq 1$.

In our next result, we use the following notion of sparse form domination of $\ell^p$ type: \begin{equation}\label{eq:introsparseformellp}
  \int_{\R^d}\!|T\vec{f}|^p|g|\,\mathrm{d}x\lesssim \sum_{Q\in\mathcal{S}}\Big(\prod_{j=1}^{m}\langle f_j\rangle_{1,Q}\Big)^p\langle g\rangle^p_{1,Q}|Q|.
\end{equation}
While any known example of an operator satisfying \eqref{eq:introsparseform} also satisfies \eqref{eq:introsparseformellp} for $p\in (0,1]$, it is not clear if this implication always holds, see \cite[Conjecture~6.2]{LN20}. 
Regardless, this hypothesis is still very general and is satisfied by our operators of interest. In particular, one has the following pointwise sparse bound from \cite{CR16, LN15, DHL18}: if $T$ is a (maximal truncation of) a multilinear Calder\'on-Zygmund operator (with Dini-continuous kernel) and $f_1,\ldots,f_m \in L_0^{\infty}(\R^d)$, then there exists a sparse collection $\mathcal{S}$ such that 
\begin{equation}\label{eq:introsparsepointwise}
   |T\vec{f}(x)|\lesssim\sum_{Q \in \mathcal{S}}\prod_{j=1}^m \langle f_j\rangle_{1,Q}\ind_Q(x)=: A_{\mc{S}}\vec{f}(x)
\end{equation}
for almost all $x \in \R^d$. 
Note that if an $m$-sublinear operator $T$ satisfies \eqref{eq:introsparsepointwise}, then $T$ necessarily satisfies \eqref{eq:introsparseformellp} for $p \in (0,1]$. 

\begin{TheoremLetter}\label{MLMultiplierWeakType}
    Let $T$ be an $m$-sublinear operator satisfying sparse form domination, let $\vec{p} \in [1,\infty]^m$, and let $p\in [\tfrac{1}{m},\infty)$ satisfy $\frac{1}{p} = \sum_{j=1}^m\frac{1}{p_j}$. If $w \in A_{\vec{p}}$ and $p\geq 1$, then
    \begin{equation}\label{eq:eq:MLMultiplierWeakType}
        \|T(\vec{f}/\vec{w})w\|_{L^{p,\infty}(\R^d)}\lesssim [w^p]_{\text{FW}}[\vec{w}]_{\vec{p}}\|\vec{f}\|_{L^{\vec{p}}(\R^d)}
    \end{equation}
    for all $\vec{f}\in L^{\vec{p}}(\R^d)$. Moreover, if $p<1$ and $T$ satisfies sparse form domination of $\ell^p$ type, then \eqref{eq:eq:MLMultiplierWeakType} remains valid.
\end{TheoremLetter}

We discuss our results when applied in the linear setting -- in this case, Theorem~\ref{thm:E} gives 
$$
    \|Tf\|_{L^{p,\infty}(\R^d,w)}\lesssim [w]_{\text{FW}}^{1/p'}[w]_{A_p}^{1/p}\|f\|_{L^p(\R^d,w)}
$$
for $p>1$ and $w \in A_p$, and 
$$
    \|Tf\|_{L^{1,\infty}(\R^d,w)}\lesssim [w]_{\text{FW}}[w]_{A_1}\|f\|_{L^1(\R^d,w)}
$$
for $w \in A_1$, while Theorem~\ref{MLMultiplierWeakType} implies 
$$
    \|T(fw^{-1/p})w^{1/p}\|_{L^{p,\infty}(\R^d)}\lesssim [w]_{\text{FW}}[w]_{A_p}^{1/p}\|f\|_{L^p(\R^d)}
$$ 
for $p\ge 1$ and $w \in A_p$. Note that the $p>1$ case of Theorem~\ref{thm:E} and the $p=1$ case of Theorem~\ref{MLMultiplierWeakType} recover the known sharp quantitative linear bounds. However, the second above estimate improves to a sharp dependence of $(1+\log [w]_{\text{FW}})[w]_{A_1}$, and the third estimate holds with the asymptotically smaller constant $(1+\log [w]_{A_p})^{1/p}[w]_{A_p}^{1+1/p^2}$, see the very recent paper \cite{LLORr24}. 

\begin{remark}
    \emph{The proofs of these improvements in the linear setting use the following fact: if $\mathcal{S}$ is a sparse collection, $f\in L^1_{\text{loc}}(\R^d)$, and $\lambda >0$, then for each 
    $$
    Q \in \mathcal{G}:=\{Q \in \mathcal{S}: \lambda < \langle f\rangle_{1,Q}\leq 2\lambda\}
    $$ 
    there exists $G_Q \subseteq Q$ such that $\{G_Q\}_{Q \in \mathcal{G}}$ is a disjoint collection and $\int_Q |f|\,dx \lesssim \int_{G_Q} |f|\,dx$, see \cite[p.\,68,\,eq.\,(3.4)]{DLR16}, \cite[Lemma~3.2]{CRr20}, and \cite[Lemma~2.3]{LLORr24}. 
We do not believe this result extends to the multilinear setting, and thus we do not expect to recover the case $m=1$. It remains an interesting open question to determine whether or not our bounds are sharp for $m>1$.}
\end{remark}


Theorem~\ref{MLMultiplierWeakType} is a direct generalization of the quantitative multiplier weak-type bound of \cite{CS23} to the multilinear setting. Aside from the qualitative result in the endpoint case $\vec{p}=\vec{1}$ of \cite{LOP19}, Theorem~\ref{MLMultiplierWeakType} is the only known multiplier weak-type bound for multilinear Calder\'on-Zygmund operators. We emphasize that Theorem~\ref{MLMultiplierWeakType} holds for all $\vec{p} \in [1,\infty]^m$ with $p \in [\tfrac{1}{m},\infty)$ for $m$-sublinear operators satisfying pointwise sparse domination.

To place Theorem~\ref{thm:E} into context, we note that one has a quantitative version of \eqref{MultilinearWeakType} in terms of the multilinear Fujii-Wilson condition
\[
[\vec{w}]_{\text{FW}}^{\vec{p}}:=\sup_Q\Big(\int_Q\prod_{j=1}^m w_j^{\frac{p}{p_j}}\,\mathrm{d}x\Big)^{-\frac{1}{p}}\Big(\int_Q\!M_{\vec{p}}\big(w_1^{\frac{1}{p_1}}\ind_Q,\ldots,w_m^{\frac{1}{p_m}}\ind_Q\big)^p\,\mathrm{d}x\Big)^{\frac{1}{p}}
\]
introduced in \cite{Zor19}. 
When $m=1$, $[\vec{w}]_{\text{FW}}^{\vec{p}}=[w]_{\text{FW}}^{1/p}$. It was shown in \cite[Theorem~1.11]{Zor19} that if $T$ satisfies sparse form domination and $\vec{p} \in (1,\infty)^m$ with $p>1$, then 
\begin{equation}\label{eq:zor1}
\|T\vec{f}\|_{L^{p,\infty}_w(\R^d)}
\lesssim B_{\vec{w}} [\vec{w}]_{\vec{p}}\|\vec{f}\|_{L^{\vec{p}}_{\vec{w}}(\R^d)}
\end{equation}
for 
$B_{\vec{w}}:= \max_{k \in \{1,\ldots, m\}} B_{\vec{w}}^k$, where 
\[
B_{\vec{w}}^k:=
\big[(w_1^{-p_1'},\ldots,w_{k-1}^{-p_{k-1}'},w^p,w_{k+1}^{-p_{k+1}'},\ldots,w_m^{-p_m'})\big]^{(p_1,\ldots,p_{k-1},p',p_{k+1},\ldots,p_m)}_{\text{FW}}.
\]
Taking $m=1$, the estimate \eqref{eq:zor1} reduces to the sharp bound
\[
\|Tf\|_{L^{p,\infty}_w(\R^d)}\lesssim [w^p]_{\text{FW}}^{\frac{1}{p'}}[w^p]_{A_p}^{\frac{1}{p}}\|f\|_{L^p_w(\R^d)}\lesssim [w^p]_{A_p}\|f\|_{L^p_w(\R^d)}
\]
for $p>1$; however, using 
\cite[Lemma~1.6]{Zor19} or \cite[Proposition~3.3.3 (ii)]{Ni20}, one has
\[
B_{\vec{w}}^k \lesssim [\vec{w}]_{\vec{p}}^{\max\big(\frac{p_1'}{p_1},\ldots,\frac{p_{k-1}'}{p_{k-1}},\frac{p}{p'},\frac{p_{k+1}'}{p_{k+1}},\ldots,\frac{p_m'}{p_m}\big)}
\]
and, hence,
\eqref{eq:zor1} only gives 
\begin{equation}\label{eq:weakstrongboundintro}
\|T\vec{f}\|_{L^{p,\infty}_w(\R^d)}
\lesssim [\vec{w}]^{\max(p,p_1',\ldots,p_m')}_{\vec{p}}\|\vec{f}\|_{L^{\vec{p}}_{\vec{w}}(\R^d)}
\end{equation}
for $\vec{p} \in (1,\infty)^m$ with $p>1$ and $m>1$. 
Note that \eqref{eq:weakstrongboundintro} coincides with the bound inherited from the strong-type bound \eqref{MultilinearSharpStrongType} 
and, hence, \eqref{eq:zor1} does not yield an improvement with respect to $[\vec{w}]_{\vec{p}}$ for multilinear operators. 

An improvement of  \eqref{eq:weakstrongboundintro} 
was very recently obtained for operators satisfying pointwise sparse domination in the bilinear setting in \cite{Zh23}. For $\vec{p} \in (1,\infty)^2$ with $p>1$, one has
\begin{equation}\label{eq:zh23main}
\begin{split}
\|A_{\mc{S}}(f_1,f_2)\|_{L^{p,\infty}_w(\R^d)}
&\lesssim C_{\vec{w}}
[\vec{w}]_{\vec{p}}
\|f_1\|_{L^{p_1}_{w_1}(\R^d)}\|f_2\|_{L^{p_2}_{w_2}(\R^d)}\\
&\lesssim[\vec{w}]_{\vec{p}}^\alpha
\|f_1\|_{L^{p_1}_{w_1}(\R^d)}\|f_2\|_{L^{p_2}_{w_2}(\R^d)},
\end{split}
\end{equation}
for all $f_j \in L^{p_j}_{w_j}(\R^d)$ with $j=1,2$, where
\[
C_{\vec{w}}:=\max\Big(\min\big([w^p]_{\text{FW}},[w_2^{-p_2'}]_{\text{FW}}\big)^{\frac{1}{p_1'}}, \min\big([w_1^{-p_1'}]_{\text{FW}},[w^p]_{\text{FW}}\big)^{\frac{1}{p_2'}}\Big) 
\]
and
\[
\alpha:=1+\max\big(\min\big(\tfrac{p}{p_1'},\tfrac{p_2'}{p_1'}\big),\min\big(\tfrac{p_1'}{p_2'},\tfrac{p}{p_2'}\big)\big).
\]
This improves \eqref{eq:weakstrongboundintro} for certain $\vec{p}$, see \cite{Zh23}. The second case of Theorem~\ref{thm:E} extends \eqref{eq:zh23main} to general $m$ for operators satisfying sparse form domination. We emphasize that our intermediate steps were developed independently of the work in \cite{Zh23}, but in the end, we rely on their clever application of the sharp reverse H\"older inequality. On the other hand, we showcase a unique perspective in relation to the multilinear Fujii-Wilson constant of \cite{Zor19}, see the above discussion and Appendix~\ref{app:A}. Additionally, our argument provides a new quantitative bound for $p\leq 1$ through the use of a new good-$\lambda$ inequality. 


The proof of \eqref{eq:zh23main} in \cite{Zh23} relies on an equivalence 
with the testing condition
\[
\int_Q\!A_{\mc{S}}(f_1,f_2)v\,\mathrm{d}x\lesssim_{\vec{w}} \|f_1\|_{L^{p_1}_{w_1}(\R^d)}\|f_2\|_{L^{p_2}_{w_2}(\R^d)}v(Q)^{\frac{1}{p'}}
\]
for all $Q\in\mc{S}$, where $v:=w^p$. We show that this equivalence holds for general $m$ and can be improved to a local testing condition that coincides with the 
condition of \cite{LSU09} in the case $m=1$. We give a short, original proof of this result below in Theorem~\ref{thm:D}, and then prove Theorem~\ref{thm:E} in the case $\vec{p}\in(1,\infty)^m$, $p>1$ by verifying this testing condition. Moreover, our general bound, which applies in the case $p\leq 1$, uses the new good-$\lambda$ inequality of Theorem~\ref{thm:goodlambdav} that holds independent interest.

The paper is organized as follows. In Section \ref{Preliminaries}, we collect relevant notation, discuss the dyadic structure, and state Kolmogorov's inequality. In Section \ref{TheoremAPart1}, we establish a local testing theorem for operators satisfying sparse form domination and prove Theorem \ref{thm:E} in the case $p>1$. In Section \ref{TheoremAPart2}, we obtain a linearization, prove a good-$\lambda$ inequality, and verify Theorem \ref{thm:E} in the case $p \leq 1$. We prove Theorem \ref{MLMultiplierWeakType} in Section \ref{TheoremBSection}. We end with 
a concluding discussion in Appendix \ref{app:A}.

\section{Preliminaries}\label{Preliminaries}

\subsection{Notation}
Fix positive integers $d$ and $m$. For $A,B>0$, we write $A \lesssim B$ if there exists $C>0$ (which possibly depends on $d$, $m$, $\vec{p}$, or $T$) such that $A \leq CB$, write $A \eqsim B$ if $A \lesssim B \lesssim A$, and write $A\lesssim_{\alpha}B$ if the implicit constant may depend on a parameter $\alpha$. 
\begin{itemize}
    \item $L^1_{\text{loc}}(\R^d)$ is the space of locally integrable functions on $\R^d$;
    \item We call $w$ a weight if $w \in L^1_{\text{loc}}(\R^d)$ and $w(x)>0$ for almost all $x \in \R^d$;
    \item For a weight $w$ and $A\subseteq \R^d$, we write $w(A):=\int_{\R^d} w\,dx$ and write $|A|$ when $w \equiv 1$;
    \item A cube is a set in $\R^d$ of the form $\prod_{j=1}^d[a_j,b_j)$ with $b_j-a_j$ equal for all $j \in \{1,\ldots,d\}$; 
    \item For a collection of cubes $\mc{P}$ and a cube $Q$, we write $\mc{P}(Q):= \{Q' \in \mc{P}: Q'\subseteq Q\}$;
    \item For a measurable $f$, $p>0$, a weight $w$, and a cube $Q$, we write 
    $$
    \langle f\rangle_{p,Q}^w:=\left(\frac{1}{w(Q)}\int_Q\!|f|^pw\,\mathrm{d
}x\right)^{1/p},
$$ 
we omit the superscript $w$ when $w\equiv 1$, and we define $\langle f\rangle_{\infty,Q}:=\esssup_{x\in Q}|f(x)|$;
    \item For a weight $w$ and $p >0$, we write
$$
    \|f\|_{L^p_w(\R^d)}:=\|fw\|_{L^p(\R^d)} \quad \text{and}\quad \|f\|_{L^{p,\infty}_w(\R^d)}:=\sup_{\lambda>0}\|\lambda\ind_{\{|f|>\lambda\}}\|_{L^p_w(\R^d)};
$$
    \item $L_{0}^{\infty}(\R^d)$ is the space of essentially bounded functions on $\R^d$ with compact support;
    \item For $p \in [1,\infty]$, the H\"older conjugate $p'$ is defined by $\tfrac{1}{p}+\tfrac{1}{p'}=1$;
    \item We write $\sup_Q$ to indicate a supremum taken over all cubes $Q$ in $\R^d$;
    \item For a weight $w$ and $p \in (1,\infty)$, we write $w \in A_p$ if 
$$
    [w]_{A_p}:=\sup_Q\langle w \rangle_{1,Q}\langle w^{1-p'}\rangle_{1,Q}^{p-1}<\infty
$$
    and $w \in A_1$ if 
$$    
    [w]_{A_1}:= \sup_Q \langle w\rangle_{1,Q}\langle w^{-1}\rangle_{\infty,Q}<\infty;
$$
    \item For a collection of cubes $\mathcal{P}$ and a weight $w$, we write 
    $$
        M^{\mathcal{P},w}f := \sup_{Q\in\mathcal{P}}\langle f\rangle_{1,Q}^w\ind_Q,
    $$
    where we omit the superscripts when either $\mathcal{P}$ is the collection of all cubes or $w\equiv 1$;
    \item For a weight $w$, we write $w \in A_{\text{FW}}$ if 
$$
    [w]_{\text{FW}}:=\sup_{Q}\frac{1}{w(Q)}\int_Q\!M(w\ind_Q)\,\mathrm{d}x<\infty;
$$
    \item For measurable $f_1,\ldots,f_m$, we write $\vec{f}:=(f_1,\ldots,f_m)$; 
    \item For measurable $\vec{f}$ and weights $\vec{w}$, we write $\vec{f}/\vec{w}:=(f_1w_1^{-1},\ldots,f_mw_m^{-1});$
    \item For $p_1,\ldots, p_m \in (0,\infty]$, we write $\vec{p}:=(p_1,\ldots,p_m)$;
    \item For $\vec{p} \in (0,\infty]^m$, $p$ is defined by $\frac{1}{p}=\sum_{j=1}^{m}\frac{1}{p_j}$;
    \item For weights $\vec{w}$, we write $w:=\prod_{j=1}^m w_j$;
    \item For weights $\vec{w}$ and $\vec{p} \in (0,\infty]^m$, we write 
    $$
        L^{\vec{p}}_{\vec{w}}(\R^d):= \prod_{j=1}^{m}L^{p_j}_{w_j}(\R^d) \quad\text{and}\quad \|\vec{f}\|_{L^{\vec{p}}_{\vec{w}}(\R^d)}:=\prod_{j=1}^m\|f_j\|_{L^{p_j}_{w_j}(\R^d)};
    $$
\item For weights $\vec{w}$ and $\omega$ and $\vec{p} \in [1,\infty)^m$, we write $(\vec{w},\omega) \in A_{\vec{p}}$ if
$$
   [\vec{w},\omega]_{\vec{p}}:=\sup_Q\langle \omega\rangle_{p,Q}\prod_{j=1}^m\langle w_j^{-1}\rangle_{p_j',Q}<\infty;
$$
    \item For weights $\vec{w}$ and $\vec{p}\in[1,\infty)^m$, we write $\vec{w}\in A_{\vec{p}}$ if
\[
[\vec{w}]_{\vec{p}}:=[\vec{w},w]_{\vec{p}}<\infty;
\]
\item For $\vec{p} \in (0,\infty]^m$, we write 
\[
    M_{\vec{p}}\vec{f}:=\sup_Q\prod_{j=1}^m\langle f_j\rangle_{p_j,Q}\ind_Q;
\]
\item For weights $\vec{w}$, we write
\[
[\vec{w}]_{\text{FW}}^{\vec{p}}:=\sup_Q\Big(\int_Q\prod_{j=1}^m w_j^{\frac{p}{p_j}}\,\mathrm{d}x\Big)^{-\frac{1}{p}}\Big(\int_Q\!M_{\vec{p}}\big(w_1^{\frac{1}{p_1}}\ind_Q,\ldots,w_m^{\frac{1}{p_m}}\ind_Q\big)^p\,\mathrm{d}x\Big)^{\frac{1}{p}};
\]
\item For a collection of cubes $\mc{P}$, we write 
\[
A_{\mc{P}}\vec{f}:=\sum_{Q\in\mc{P}}\prod_{j=1}^m\langle f_j\rangle_{1,Q}\ind_Q \quad\text{and}\quad M^{\mc{P}}\vec{f}:=\sup_{Q\in\mc{P}}\prod_{j=1}^m\langle f_j\rangle_{1,Q}\ind_Q;
\]
\item We denote by $\|T\|_{\mathcal{X}\rightarrow \mathcal{Y}}$ the smallest constant $C>0$ such that
$$
    \|Tf\|_{\mathcal{Y}}\leq C\|f\|_{\mathcal{X}}
$$
for all $f \in \mathcal{X}$.
\end{itemize}
We note that if $\vec{w}\in A_{\vec{p}}$, then $w^p\in A_{mp}\subseteq A_{\text{FW}}$ and
\[
[w^p]_{\text{FW}}\lesssim [w^p]_{A_{mp}}\leq[\vec{w}]^p_{\vec{p}}. 
\]

\subsection{Dyadic analysis}
We call $\mc{D}$ a dyadic grid if there exists $\alpha\in\{0,\tfrac{1}{3},\tfrac{2}{3}\}^d$ for which $\mc{D}=\mc{D}^\alpha$, where
\[
\mc{D}^\alpha:=\{2^{-j}\big([0,1)^d+\alpha+k\big):j \in \Z, \,\, k\in\Z^d\}.
\]
The $3^d$-lattice theorem states that for each cube $Q$, there exists $\alpha\in\{0,\tfrac{1}{3},\tfrac{2}{3}\}^d$ and $\widetilde{Q}\in\mc{D}^\alpha$ such that $Q\subseteq\widetilde{Q}$ and $|\widetilde{Q}|\leq 6^d|Q|$, see \cite{LN15}.

For $\eta\in(0,1)$, a collection of cubes $\mc{S}$ is called $\eta$-sparse if for each $Q \in \mathcal{S}$ there exists $E_Q\subseteq Q$ such that $|E_Q|\geq \eta |Q|$ and $\{E_Q\}_{Q \in \mathcal{S}}$ is a disjoint collection. If $\eta=\tfrac{1}{2}$, then we simply say that $\mc{S}$ is sparse.
If $\mc{S}$ is sparse, then the $3^d$-lattice theorem implies that there exist $\tfrac{1}{2\cdot 6^d}$-sparse collections $\mc{S}^\alpha\subseteq\mc{D}^\alpha$ such that
\[
A_{\mc{S}}\vec{f}(x)\lesssim\sum_{\alpha\in\{0,\tfrac{1}{3},\tfrac{2}{3}\}^d}A_{\mc{S}^\alpha}\vec{f}(x).
\]

For a dyadic grid $\mc{D}$, we define $A_{\text{FW}}(\mc{D})$ in the same way as $A_{\text{FW}}$, but with the supremum taken over $\mc{D}$ rather than over all cubes. We will need the following sharp reverse H\"older inequality for weights satisfying the Fujii-Wilson condition.
\begin{theorem}[\cite{HP13}]\label{thm:sharprh}
Let $\mc{D}$ be a dyadic grid, $w\in A_{\text{FW}}(\mc{D})$, and $Q\in\mc{D}$. If $r\in(1,\infty)$ satisfies $r'\geq 2^{d+1}[w]_{\text{FW}}$, then
\[
\langle w\rangle_{r,Q}\leq 2\langle w\rangle_{1,Q}.
\]
\end{theorem}

\subsection{Kolmogorov's lemma}
We frequently appeal to Kolmogorov's lemma, which, for a $\sigma$-finite measure space $(\Omega,\mu)$ and $p \in(0,\infty)$, states that $f\in L^{p,\infty}(\Omega,\mu)$ if and only if there exists $C>0$ such that for all $E\subseteq\Omega$ with $0<\mu(E)<\infty$ and all $0<\theta<r$, one has
\begin{equation}\label{eq:kolmogorov}
\int_E\!|f|^\theta\,\mathrm{d}\mu\leq\tfrac{p}{p-\theta}C^\theta\mu(E)^{1-\frac{\theta}{p}},
\end{equation}
in which case the optimal constant $C>0$ satisfies
\[
C\leq \|f\|_{L^{p,\infty}(\Omega,\mu)}\leq \big(\tfrac{p}{p-\theta}\big)^{\frac{1}{\theta}}C.
\]
We also use the variant that asserts $f\in L^{p,\infty}(\Omega,\mu)$ if and only if there exists $C>0$ such that for each $E\subseteq\Omega$ with $0<\mu(E)<\infty$ there exists a $E'\subseteq E$ with $\mu(E')\geq\tfrac{1}{2}\mu(E)$ and
\begin{equation}\label{eq:kolmogorov2}
\int_{E'}\!|f|\,\mathrm{d}\mu\leq C\mu(E)^{1-\frac{1}{p}},
\end{equation}
in which case the optimal constant $C$ satisfies
\[
2^{-\frac{1}{p}}C\leq \|f\|_{L^{p,\infty}(\Omega,\mu)}\leq 2C,
\]
see \cite[Exercise~1.4.14]{Gr14a}.

\section{Proof of Theorem~\ref{thm:E} in the case $\vec{p}\in(1,\infty)^m$, $p>1$}\label{TheoremAPart1}
\subsection{Local testing for sparse form domination}
The proof of Theorem~\ref{thm:E} in the case $p>1$ uses the following local testing condition.
\begin{theorem}\label{thm:D}
Let $T$ be an $m$-sublinear operator satisfying sparse form domination, let $\vec{p}\in (1,\infty)^m$, and let $p>1$ satisfy $\tfrac{1}{p}=\sum_{j=1}^m\tfrac{1}{p_j}$. If $\vec{w}$ and $\omega$ are weights, then 
\begin{equation}\label{eq:thmdtesting}
\|T\|_{L^{\vec{p}}_{\vec{w}}(\R^d)\to L^{p,\infty}_\omega(\R^d)}\lesssim\sup_{\mc{S}}\sup_{Q_0\in\mc{S}}\sup_{\substack{\|f_j\|_{L^{p_j}_{w_j}(Q_0)}=1\\ j\in\{1,\ldots,m\}}}v(Q_0)^{-\frac{1}{p'}}\int_{Q_0}\!A_{\mc{S}(Q_0)}(\vec{f}) v\,\mathrm{d}x,
\end{equation}
where $v:=w^p$ and the first supremum is taken over all finite $\tfrac{1}{2\cdot 6^d}$-sparse collections contained in some dyadic grid. 

\end{theorem}
\begin{remark}
One can actually show that 
\[
\|A_{\mc{S}}\|_{L^{\vec{p}}_{\vec{w}}(\R^d)\to L^{p,\infty}_\omega(\R^d)}\eqsim \sup_{Q_0\in\mc{S}}\sup_{\substack{\|f_j\|_{L^{p_j}_{w_j}(Q_0)}=1\\ j\in\{1,\ldots,m\}}}v(Q_0)^{-(\frac{1}{q}-\frac{1}{p})}\Big(\int_{Q_0}\!(A_{\mc{S}(Q_0)}\vec{f})^q v\,\mathrm{d}x\Big)^{\frac{1}{q}}
\]
for all $1\leq q\leq p<\infty$. Additionally, if $k\in\{1,\ldots,m\}$, then
\[
\int_{Q_0}\!(A_{\mc{S}(Q_0)}\vec{f}) v\,\mathrm{d}x=\int_{Q_0}\!|f_k|A_{\mc{S}(Q_0)}(f_1,\ldots,f_{k-1},v,f_{k+1},\ldots,f_m)\,\mathrm{d}x,
\]
so the term inside the supremum on the right-hand side of \eqref{eq:thmdtesting} is equal to
\[
\sup_{\substack{\|f_j\|_{L^{p_j}_{w_j}(Q_0)}=1\\ j\in\{1,\ldots,m\}\setminus\{k\}}}v(Q_0)^{-\frac{1}{p'}}\Big\|A_{\mc{S}(Q_0)}(f_1,\ldots,f_{k-1},v,f_{k+1},\ldots,f_m)\Big\|_{L_{w_k^{-1}}^{p_k'}(Q_0)}.
\]
When $m=1$, this gives the local testing condition from \cite{LSU09}:
\[
\|A_{\mc{S}}\|_{L^p_w(\R^d)\to L^{p,\infty}_w(\R^d)}\eqsim\sup_{Q_0\in\mc{S}}v(Q_0)^{-\frac{1}{p'}}\Big\|\sum_{Q\in\mc{S}(Q_0)}\langle v\rangle_{1,Q}\ind_Q\Big\|_{L_{w^{-1}}^{p'}(Q_0)}.
\]
\end{remark}

\begin{proof}[Proof of Theorem~\ref{thm:D}]
Let $E\subseteq\R^d$ with $0<v(E)<\infty$. Then there is a sparse collection $\mc{S}$ such that
\begin{equation}\label{eq:thmd1}
\int_E\!|T(\vec{f})|v\,\mathrm{d}x\lesssim \sum_{Q\in\mc{S}}\Big(\prod_{j=1}^m\langle f_j\rangle_{1,Q}\Big)\langle v\ind_E\rangle_{1,Q}|Q|.
\end{equation}
By the $3^d$-lattice theorem, we may assume that $\mc{S}$ is $\tfrac{1}{2\cdot 6^d}$-sparse and contained in a dyadic grid $\mc{D}$, and by monotone convergence we may assume that $\mc{S}$ is finite.

Fix $\lambda>0$, let
\[
\mc{S}_\lambda:=\{Q\in\mc{S}:\langle\ind_E\rangle^v_{1,Q}>\lambda\},
\]
and let $\mc{S}^\ast_\lambda$ denote the maximal cubes in $\mc{S}_\lambda$. Note that 
\[
\sum_{Q_0\in\mc{S}_\lambda}\ind_{Q_0}=\ind_{\{M^{\mc{S}_\lambda,v}(\ind_E)>\lambda\}}\leq\ind_{\{M^{\mc{D},v}(\ind_E)>\lambda\}}.
\]
Hence, denoting the right-hand side of \eqref{eq:thmdtesting} by $\mc{M}$, we have
\begin{align*}
\sum_{Q\in\mc{S}_\lambda}\Big(\prod_{j=1}^m\langle f_j\rangle_{1,Q}\Big)v(Q)
&=\sum_{Q_0\in\mc{S}^\ast_\lambda}\sum_{Q\in\mc{S}(Q_0)}\Big(\prod_{j=1}^m\langle f_j\rangle_{1,Q}\Big)v(Q)\\
&\leq\mc{M}\sum_{Q_0\in\mc{S}^\ast_\lambda}\prod_{j=1}^m\Big(\int_{Q_0}\!|f_j|^{p_j}w_j^{p_j}\,\mathrm{d}x\Big)^{\frac{1}{p_j}}v(Q_0)^{\frac{1}{p'}}\\
&\leq\mc{M}\prod_{j=1}^m\Big(\sum_{Q_0\in\mc{S}^\ast_\lambda} \int_{Q_0}\!|f_j|^{p_j}w_j^{p_j}\,\mathrm{d}x\Big)^{\frac{1}{p_j}}\Big(\sum_{Q_0\in\mc{S}^\ast_\lambda}v(Q_0)\Big)^{\frac{1}{p'}}\\
&\leq\mc{M}\Big(\prod_{j=1}^m\|f_j\|_{L^{p_j}_{w_j}(\R^d)}\Big)v\big(\{M^{\mc{D},v}(\ind_E)>\lambda\}\big)^{\frac{1}{p'}},
\end{align*}
where in the second to last step we used H\"older's inequality. Thus,
\begin{align*}
\sum_{Q\in\mc{S}}\Big(\prod_{j=1}^m\langle f_j\rangle_{1,Q}\Big)\langle v\ind_E\rangle_{1,Q}|Q|
&=\sum_{Q\in\mc{S}}\Big(\prod_{j=1}^m\langle f_j\rangle_{1,Q}\Big)\Big(\int_0^{\langle\ind_E\rangle^v_{1,Q}}\!\,\mathrm{d}\lambda\Big)v(Q)\\
&=\int_0^\infty \sum_{Q\in\mc{S}_\lambda}\Big(\prod_{j=1}^m\langle f_j\rangle_{1,Q}\Big)v(Q)\,\mathrm{d}\lambda\\
&\leq\mc{M}\Big(\prod_{j=1}^m\|f_j\|_{L^{p_j}_{w_j}(\R^d)}\Big)\|M^{\mc{D},v}(\ind_E)\|_{L^{p',1}(\R^d,v)}\\
&\lesssim\mc{M}\Big(\prod_{j=1}^m\|f_j\|_{L^{p_j}_{w_j}(\R^d)}\Big)v(E)^{\frac{1}{p'}}.
\end{align*}
Combining this with \eqref{eq:thmd1}, the result follows from Kolmogorov's lemma.
\end{proof}

\subsection{Proof of Theorem~\ref{thm:E} in the case $ \vec{\bf{p}}\bf\in(1,\infty)^m$, $\bf p>1$}
We need several lemmata.
We first need the following application of Kolmogorov's lemma:
\begin{lemma}\label{lem:multilineartesting3}
If $\mc{S}$ is a sparse collection in a dyadic grid $\mc{D}$ and $\alpha_1,\ldots,\alpha_m\in[0,1)$ with $\sum_{j=1}^m\alpha_j<1$, then 
\[
\sum_{\substack{Q'\in\mc{S}\\ Q'\subseteq Q}}\Big(\prod_{j=1}^m\langle g_j\rangle_{1,Q'}^{\alpha_j}\Big)|Q'|\lesssim\Big(\prod_{j=1}^m\langle g_j\rangle_{1,Q}^{\alpha_j}\Big)|Q|
\]
for all $g_1,\ldots,g_m\in L^1_{\text{loc}}(\R^d)$ and all $Q\in\mc{D}$.
\end{lemma}
\begin{proof}
Pick $\theta_1,\ldots,\theta_m\in[0,1]$ with $\theta_j>\alpha_j$ and $\sum_{j=1}^m\theta_j=1$. 
Using the sparseness condition, H\"older's inequality, and Kolmogorov's lemma \eqref{eq:kolmogorov}, we have
\begin{align*}
\sum_{\substack{Q'\in\mc{S}\\ Q'\subseteq Q}}\Big(\prod_{j=1}^m\langle g_j\rangle_{1,Q'}^{\alpha_j}\Big)|Q'|&\lesssim\int_Q\!\prod_{j=1}^m (M^{\mc{D}(Q)}g_j)^{\alpha_j}\,\mathrm{d}x
\leq\prod_{j=1}^m\Big(\int_Q\!(M^{\mc{D}(Q)}g_j)^{\frac{\alpha_j}{\theta_j}}\,\mathrm{d}x\Big)^{\theta_j}\\
&\leq\prod_{j=1}^m\big(\tfrac{1}{1-\frac{\alpha_j}{\theta_j}}\big)^{\theta_j}\|g_j\|_{L^1(Q)}^{\alpha_j}|Q|^{\theta_j-\alpha_j}
\eqsim\Big(\prod_{j=1}^m\langle g_j\rangle_{1,Q}^{\alpha_j}\Big)|Q|.
\end{align*}
This proves the result.
\end{proof}

The next lemma uses \cite[Proposition~2.2]{COV04} which gives 
\begin{equation}\label{eq:covlemma}
\Big\|\sum_{Q\in\mc{F}}a_Q\ind_Q\Big\|_{L^q(\R^d,v)}\eqsim\Big(\sum_{Q\in\mc{F}}\Big(\frac{1}{v(Q)}\sum_{\substack{Q'\in\mc{F}\\ Q'\subseteq Q}}a_{Q'}v(Q)\Big)^{q-1}a_Qv(Q)\Big)^{\frac{1}{q}}.
\end{equation}
for all $q\in[1,\infty)$, weights $v$, collections $\mc{F}$ in a dyadic grid, and $\{a_Q\}_{Q\in\mc{F}}\subseteq [0,\infty)$. 

\begin{lemma}\label{lem:thme}
If $\mc{S}$ is a sparse collection in a dyadic grid $\mc{D}$, $\vec{p}\in(1,\infty)^m$ with $p>1$, and $(\vec{w},\omega)\in A_{\vec{p}}$, then 
\[
\Big\|\sum_{Q\in\mc{S}}\prod_{j=0}^{m-1}\langle v_j\rangle_{1,Q}\ind_Q\Big\|_{L^{p_m'}(\R^d,v_m)}\lesssim [\vec{w},\omega]_{\vec{p}}\Big(\sum_{Q\in\mc{S}}\Big(\prod_{j=0}^{m-1}\langle v_j\rangle_{1,Q}^{\frac{p_m'}{p_j}}\Big)|Q|\Big)^{\frac{1}{p_m'}},
\]
where $p_0:=p'$, $v_0:=\omega^p$, and $v_j:=w_j^{-p_j'}$.
\end{lemma}
\begin{proof}
Let $\gamma:=\min\{p_0',\ldots,p_m'\}$ and observe
\[
m+1\geq\sum_{j=0}^m \frac{\gamma}{p_j'}=\gamma m,
\]
so that $1<\gamma\leq1+\frac{1}{m}$. Since $\sum_{j=0}^m 1-\frac{\gamma}{p_j'}=m+1-m\gamma\in[0,1),
$ Lemma~\ref{lem:multilineartesting3} gives
\begin{equation}\label{eq:C21}
\begin{split}
\frac{1}{v_m(Q)}\sum_{\substack{Q'\in\mc{S}\\ Q'\subseteq Q}}\Big(\prod_{j=0}^m\langle v_j\rangle_{1,Q'}\Big)|Q'|
&\leq[\vec{w},\omega]^\gamma_{\vec{p}}\frac{1}{v_m(Q)}\sum_{\substack{Q'\in\mc{S}\\ Q'\subseteq Q}}\Big(\prod_{j=0}^m\langle v_j\rangle^{1-\frac{\gamma}{p_j'}}_{1,Q'}\Big)|Q'|\\
&\lesssim[\vec{w},\omega]^\gamma_{\vec{p}}\frac{1}{v_m(Q)}\Big(\prod_{j=0}^m\langle v_j\rangle^{1-\frac{\gamma}{p_j'}}_{1,Q}\Big)|Q|.
\end{split}
\end{equation}
Thus, by \eqref{eq:covlemma}
\begin{align*}
\Big\|\sum_{Q\in\mc{S}}\prod_{j=0}^{m-1}\langle v_j\rangle_{1,Q}\ind_Q\Big\|^{p_m'}_{L^{p_m'}(\R^d,v_m)}&\lesssim[\vec{w},\omega]^{\gamma\frac{p_m'}{p_m}}_{\vec{p}}\sum_{Q\in\mc{S}}\Big(\prod_{j=0}^{m-1}\langle v_j\rangle^{1+\frac{p_m'}{p_m}(1-\frac{\gamma}{p_j'})}_{1,Q}\Big)\langle v_m\rangle_{1,Q}^{1-\frac{\gamma}{p_m}}|Q|\\
&\leq[\vec{w},\omega]^{p_m'}_{\vec{p}}\sum_{Q\in\mc{S}}\Big(\prod_{\substack{j=0}}^{m-1}\langle v_j\rangle^{\frac{p_m'}{p_j}}_{1,Q}\Big)|Q|.
\end{align*}
This proves the assertion.
\end{proof}

\begin{proof}[Proof of Theorem~\ref{thm:E} in the case $\vec{p}\in(1,\infty)^m$, $p>1$]
Let $\mc{S}$ be a finite sparse collection in a dyadic grid $\mc{D}$ and fix $Q_0\in\mc{S}$. Let $v_j:=w_j^{-p_j'}$ and write $\lambda_{j,Q}:=\langle f_jv_j^{-1}\rangle^{v_j}_{1,Q}$ for $j\in\{1,\ldots,m\}$. For each $j\in\{1,\ldots,m\}$ and $Q\in\mc{S}(Q_0)$, let $\text{ch}_j(Q)$ denote the collection of maximal cubes $Q'\in\mc{S}(Q_0)$ satisfying $\lambda_{j,Q'}>2\lambda_{j,Q}$. Let $\mc{E}_{j,0}:=\{Q_0\}$ and recursively define
\[
\mc{E}_{j,k+1}:=\bigcup_{Q\in\mc{E}_{j,k}}\text{ch}_j(Q) \quad\text{and}\quad \mc{E}_j:=\bigcup_{k=0}^\infty\mc{E}_{j,k}.
\]
Since the sequence $\{\lambda_{j,Q}\}_{\substack{Q\in\mc{E}_j\\ Q\ni x}}$ is lacunary for each $x\in Q_0$, we have
\begin{equation}\label{eq:thme1}
\sum_{Q\in\mc{E}_j}\lambda_{j,Q}\ind_Q\leq 2 M^{\mc{E}_j,v_j}(f_jv_j^{-1}).
\end{equation}
Letting $\pi_j(Q)$ denote the smallest cube $Q'$ in $\mc{E}_j$ for which $Q\subseteq Q'$ for $Q\in\mc{S}(Q_0)$, we have
\begin{equation}\label{eq:thme2}
\lambda_{j,Q}\leq 2\lambda_{j,\pi_j(Q)}.
\end{equation}

Set $p_0:=p'$, $v_0:=w^p$, $\mu_Q:=\Big(\prod_{j=0}^m\langle v_j\rangle_{1,Q}\Big)|Q|$, write $\vec{Q}\in\mc{E}$ to mean that $\vec{Q}=(Q_1,\ldots,Q_m)$ with $Q_j\in\mc{E}_j$, and put $\pi(Q):=(\pi_1(Q),\ldots,\pi_m(Q))$. Then
\begin{equation}\label{eq:thme3}
\sum_{Q\in\mc{S}(Q_0)}\Big(\prod_{j=1}^m\langle f_j\rangle_{1,Q}\Big)v_0(Q)
=\sum_{Q\in\mc{S}(Q_0)}\Big(\prod_{j=1}^m\lambda_{j,Q}\Big)\mu_Q
=\sum_{\vec{Q}\in\mc{E}}\sum_{\substack{Q\in\mc{S}(Q_0)\\ \pi(Q)=\vec{Q}}}\Big(\prod_{j=1}^m\lambda_{j,Q}\Big)\mu_Q.
\end{equation}
If $Q\in\mc{S}(Q_0)$ satisfies $\pi(Q)=\vec{Q}$, then by the properties of the dyadic grid, $\bigcap_{j=1}^{m-1}Q_j=Q_{j_0}$ for some $j_0\in\{1,\ldots,m\}$. Moreover, as $Q\subseteq Q_{j_0}$ and $\pi_j(Q)=Q_j$, this implies that $\pi_j(Q_{j_0})=Q_j$ for all $j\in\{1,\ldots,m\}$. Thus, we have
\[
\sum_{\vec{Q}\in\mc{E}}\sum_{\substack{Q\in\mc{S}(Q_0)\\ \pi(Q)=\vec{Q}}}\Big(\prod_{j=1}^m\lambda_{j,Q}\Big)\mu_Q
\leq\sum_{j_0=1}^m\sum_{\substack{Q_j\in\mc{E}_j\\j\neq j_0}}\sum_{\substack{Q_{j_0}\in\mc{E}_{j_0}\\\pi_j(Q_{j_0})=Q_j\\j\neq j_0}}
\sum_{\substack{Q\in\mc{S}(Q_0)\\ \pi(Q)=\vec{Q}}}\Big(\prod_{j=1}^m\lambda_{j,Q}\Big)\mu_Q.
\]

By symmetry, it suffices to estimate the term with $j_0=m$. By \eqref{eq:thme2} we have
\begin{align*}
\sum_{\substack{Q_j\in\mc{E}_j\\j=1,\ldots,m-1}}&\sum_{\substack{Q_m\in\mc{E}_m\\\pi_j(Q_m)=Q_j\\j=1,\ldots,m-1}}
\sum_{\substack{Q\in\mc{S}(Q_0)\\ \pi(Q)=\vec{Q}}}\Big(\prod_{j=1}^m\lambda_{j,Q}\Big)\mu_Q\\
&\leq2^m\sum_{\substack{Q_j\in\mc{E}_j\\j=1,\ldots,m-1}}\prod_{j=1}^{m-1}\lambda_{j,Q_j}\sum_{\substack{Q_m\in\mc{E}_m\\\pi_j(Q_m)=Q_j\\j=1,\ldots,m-1}}
\lambda_{m,Q_m}\sum_{\substack{Q\in\mc{S}(Q_0)\\ \pi(Q)=\vec{Q}}}\mu_Q.
\end{align*}
Moreover, we have
\begin{align*}
\sum_{\substack{Q_m\in\mc{E}_m\\\pi_j(Q_m)=Q_j\\j=1,\ldots,m-1}}&\lambda_{m,Q_m}
\sum_{\substack{Q\in\mc{S}(Q_0)\\ \pi(Q)=\vec{Q}}}\mu_Q=\int_{Q_0}\!\sum_{\substack{Q_m\in\mc{E}_m\\\pi_j(Q_m)=Q_j\\j=1,\ldots,m-1}}\lambda_{m,Q_m}\sum_{\substack{Q\in\mc{S}(Q_0)\\ \pi(Q)=\vec{Q}}}\frac{\mu_Q}{v_m(Q)} \ind_Q v_m\,\mathrm{d}x\\
&\leq\int_{Q_0}\!\sup_{\substack{Q_m\in\mc{E}_m\\\pi_j(Q_m)=Q_j\\j=1,\ldots,m-1}}\lambda_{m,Q_m}\ind_{Q_m}\sum_{\substack{Q_m\in\mc{E}_m\\\pi_j(Q_m)=Q_j\\j=1,\ldots,m-1}}\sum_{\substack{Q\in\mc{S}(Q_0)\\ \pi(Q)=\vec{Q}}}\frac{\mu_Q}{v_m(Q)} \ind_Q v_m\,\mathrm{d}x\\
&\leq\Big\|\sup_{\substack{Q_m\in\mc{E}_m\\\pi_j(Q_m)=Q_j\\j=1,\ldots,m-1}}\lambda_{m,Q_m}\ind_{Q_m}\Big\|_{L^{p_m}(Q_0,v_m)}\Big\|\sum_{\substack{Q_m\in\mc{E}_m\\\pi_j(Q_m)=Q_j\\j=1,\ldots,m-1}}\sum_{\substack{Q\in\mc{S}(Q_0)\\ \pi(Q)=\vec{Q}}}\frac{\mu_Q}{v_m(Q)} \ind_Q\Big\|_{L^{p_m'}(Q_0,v_m)}\\
&=:\Big\|\sup_{\substack{Q_m\in\mc{E}_m\\\pi_j(Q_m)=Q_j\\j=1,\ldots,m-1}}\lambda_{m,Q_m}\ind_{Q_m}\Big\|_{L^{p_m}(Q_0,v_m)}\times I,
\end{align*}
so that by \eqref{eq:thme1}, we have
\begin{align*}
\sum_{\substack{Q_j\in\mc{E}_j\\j=1,\ldots,m-1}}&\prod_{j=1}^{m-1}\lambda_{j,Q_j}\sum_{\substack{Q_m\in\mc{E}_m\\\pi_j(Q_m)=Q_j\\j=1,\ldots,m-1}}\lambda_{m,Q_m}
\sum_{\substack{Q\in\mc{S}(Q_0)\\ \pi(Q)=\vec{Q}}}\mu_Q\\
&\leq\Big\|\sum_{\substack{Q_j\in\mc{E}_j\\j=1,\ldots,m-1}}\sum_{\substack{Q_m\in\mc{E}_m\\\pi_j(Q_m)=Q_j\\j=1,\ldots,m-1}}\lambda_{m,Q_m}\ind_{Q_m}\Big\|_{L^{p_m}(Q_0,v_m)}\Big(\sum_{\substack{Q_j\in\mc{E}_j\\j=1,\ldots,m-1}}\prod_{j=1}^{m-1}\lambda_{j,Q_j}^{p_m'}I^{p_m'}\Big)^{\frac{1}{p_m'}}\\
&\lesssim\|M^{\mc{D}(Q_0),v_m}(f_mv_m^{-1})\|_{L^{p_m}(Q_0,v_m)}\Big(\sum_{\substack{Q_j\in\mc{E}_j\\j=1,\ldots,m-1}}\prod_{j=1}^{m-1}\lambda_{j,Q_j}^{p_m'}I^{p_m'}\Big)^{\frac{1}{p_m'}}.
\end{align*}

The first factor above satisfies
\[
\|M^{\mc{D}(Q_0),v_m}(f_mv_m^{-1})\|_{L^{p_m}(Q_0,v_m)}\lesssim\|f_mv_m^{-1}\|_{L^{p_m}(Q_0,v_m)}=\|f_m\|_{L^{p_m}_{w_m}(Q_0)},
\]
so it remains to estimate the second factor. Using Lemma~\ref{lem:thme}, we have
\begin{align*}
&\Big(\sum_{\substack{Q_j\in\mc{E}_j\\j=1,\ldots,m-1}}\Big(\prod_{j=1}^{m-1}\lambda_{j,Q_j}\Big)^{p_m'}\Big\|\sum_{\substack{Q_m\in\mc{E}_m\\\pi_j(Q_m)=Q_j\\j=1,\ldots,m-1}}\sum_{\substack{Q\in\mc{S}(Q_0)\\\pi(Q)=\vec{Q}}}\frac{\mu_Q}{v_m(Q)} \ind_Q\Big\|^{p_m'}_{L^{p_m'}(Q_0,v_m)}\Big)^{\frac{1}{p_m'}}\\
&\leq[\vec{w}]_{\vec{p}}\Big(\sum_{\substack{Q_j\in\mc{E}_j\\j=1,\ldots,m-1}}\Big(\prod_{j=1}^{m-1}\lambda_{j,Q_j}\Big)^{p_m'}\sum_{\substack{Q_m\in\mc{E}_m\\\pi_j(Q_m)=Q_j\\j=1,\ldots,m-1}}\sum_{\substack{Q\in\mc{S}(Q_0)\\\pi(Q)=\vec{Q}}}\prod_{j=0}^{m-1}\langle v_j\rangle_{1,Q}^{\frac{p_m'}{p_j}}|Q|\Big)^{\frac{1}{p_m'}}.
\end{align*}
Pick $k\in\{0,\ldots,m-1\}$ such that
\[
[v_{k}]_{\text{FW}}=\min_{j\in\{0,\ldots,m-1\}}[v_j]_{\text{FW}}.
\]
Defining $r\in(1,\infty)$ by $
r'=2^{d+1}[v_{k}]_{\text{FW}}$, it follows from Theorem \ref{thm:sharprh} that
\begin{equation}\label{eq:thme4}
\langle v_{k}^r\rangle_{1,Q}^{\frac{1}{r}}\leq2\langle v_{k}\rangle_{1,Q}
\end{equation}
for all $Q\in\mc{D}$. Defining
\[
u_j:=v_j,\quad \alpha_j:=\tfrac{p_m'}{p_j},\quad\text{and}\quad \theta_j:=\tfrac{p_m'}{p_j}+\tfrac{1}{m}\tfrac{1}{r'}\tfrac{p_m'}{p_{k}}
\]
for $j\neq k$, and
\[
u_{k}:=v_{k}^r,\quad \alpha_{k}:=\tfrac{1}{r}\tfrac{p_m'}{p_{k}},\quad\text{and}\quad \theta_{k}:=\tfrac{1}{r}\tfrac{p_m'}{p_{k}}+\tfrac{1}{m}\tfrac{1}{r'}\tfrac{p_m'}{p_{k}},
\]
we have that $\alpha_j<\theta_j$ for all $j\in\{1,\ldots,m-1\}$ and $\sum_{j=1}^{m-1}\theta_j=1$. Setting $\lambda_{0,Q}=1$, it follows from H\"older's inequality and Kolmogorov's lemma \eqref{eq:kolmogorov} that
\begin{align*}
\sum_{\substack{Q_j\in\mc{E}_j\\j=1,\ldots,m-1}}&\Big(\prod_{j=1}^{m-1}\lambda_{j,Q_j}\Big)^{p_m'}\sum_{\substack{Q_m\in\mc{E}_m\\\pi_j(Q_m)=Q_j\\j=1,\ldots,m-1}}\sum_{\substack{Q\in\mc{S}(Q_0)\\\pi(Q)=\vec{Q}}}\prod_{j=0}^{m-1}\langle v_j\rangle_{1,Q}^{\frac{p_m'}{p_j}}|Q|\\
&\leq \sum_{\substack{Q_j\in\mc{E}_j\\j=1,\ldots,m-1}}\Big(\prod_{j=0}^{m-1}\lambda_{j,Q_j}\Big)^{p_m'}\sum_{\substack{Q_m\in\mc{E}_m\\\pi_j(Q_m)=Q_j\\j=1,\ldots,m-1}}\sum_{\substack{Q\in\mc{S}(Q_0)\\\pi(Q)=\vec{Q}}}\prod_{j=0}^{m-1}\langle u_j\rangle_{1,Q}^{\alpha_j}|Q|\\
&\leq\Big(\sum_{\substack{Q_j\in\mc{E}_j\\j=1,\ldots,m-1}}\prod_{j=0}^{m-1}\lambda_{j,Q_j}^{\frac{p_m'}{\theta_j}}\prod_{j=0}^{m-1}\sum_{\substack{Q_m\in\mc{E}_m\\\pi_j(Q_m)=Q_j\\j=1,\ldots,m-1}}\sum_{\substack{Q\in\mc{S}(Q_0)\\\pi(Q)=\vec{Q}}}\langle u_j\rangle_{1,Q}^{\frac{\alpha_j}{\theta_j}}|Q|\Big)^{\theta_j}\\
&\leq\prod_{j=1}^{m-1}\Big(\sum_{Q_j\in\mc{E}_j}\lambda_{j,Q_j}^{\frac{p_m'}{\theta_j}}\sum_{\substack{Q\in\mc{S}(Q_0)\\\pi_j(Q)=Q_j}}\langle u_j\rangle_{1,Q}^{\frac{\alpha_j}{\theta_j}}|Q|\Big)^{\theta_j}\Big(\sum_{Q\in\mc{S}(Q_0)}\langle u_0\rangle_{1,Q}^{\frac{\alpha_0}{\theta_0}}|Q|\Big)^{\theta_0}\\
&\lesssim\prod_{j=1}^{m-1}\Big(\tfrac{1}{1-\frac{\alpha_j}{\theta_j}}\Big)^{\theta_j}\Big(\sum_{Q_j\in\mc{E}_j}\lambda_{j,Q_j}^{\frac{p_m'}{\theta_j}}\langle u_j\rangle_{1,Q_j}^{\frac{\alpha_j}{\theta_j}}|Q_j|\Big)^{\theta_j}\langle u_0\rangle_{1,Q_0}^{\alpha_0}|Q_0|^{\theta_0}.
\end{align*}
Note that we also have 
\[
\prod_{j=0}^{m-1}\Big(\tfrac{1}{1-\frac{\alpha_j}{\theta_j}}\Big)^{\theta_j}\lesssim(r')^{\sum_{j=0}^{m-1}\theta_j}\eqsim[v_{k}]_{\text{FW}}.
\]

If $k=0$, then we use \eqref{eq:thme4} to estimate $\langle u_0\rangle_{1,Q_0}^{\alpha_0}\lesssim\langle v_0\rangle_{1,Q_0}^{\frac{p_m'}{p_0}}$ and estimate the remaining terms through \eqref{eq:thme1} with
\begin{align*}
\Big(\sum_{Q_j\in\mc{E}_j}\lambda_{j,Q_j}^{\frac{p_m'}{\theta_j}}\langle v_j\rangle_{1,Q_j}^{\frac{\alpha_j}{\theta_j}}|Q_j|\Big)^{\theta_j}&
\leq\Big(\sum_{Q_j\in\mc{E}_j}\lambda_{j,Q_j}^{\frac{p_m'}{\alpha_j}}v_j(Q_j)\Big)^{\alpha_j}\Big(\sum_{Q_j\in\mc{E}_j}|Q_j|\Big)^{\theta_j-\alpha_j}\\
&\lesssim\|M^{\mc{D}(Q_0),v_j}(f_jv_j^{-1})\|^{p_m'}_{L^{p_j}(Q_0,v_j)}|Q_0|^{\theta_j-\alpha_j}\\
&\lesssim\|f_j\|_{L^{p_j}_{w_j}(Q_0)}^{p_m'}|Q_0|^{\theta_j-\alpha_j}.
\end{align*}
As $\theta_0+\sum_{j=1}^{m-1}\theta_j-\alpha_j=\tfrac{p_m'}{p_0}$, this proves the assertion.
If $k\in\{1,\ldots,m-1\}$, we deal with the respective term through
\begin{align*}
\Big(\sum_{Q_k\in\mc{E}_k}\lambda_{k,Q_k}^{\frac{p_m'}{\theta_k}}\langle v_k^r\rangle_{1,Q_k}^{\frac{\alpha_k}{\theta_k}}|Q_k|\Big)^{\theta_k}
&\lesssim \Big(\sum_{Q_k\in\mc{E}_k}\lambda_{k,Q_k}^{\frac{p_m'}{\theta_k}}\langle v_k\rangle_{1,Q_k}^{r\frac{\alpha_k}{\theta_k}}|Q_k|\Big)^{\theta_k}\\
&\leq\Big(\sum_{Q_k\in\mc{E}_k}\lambda_{k,Q_k}^{\frac{p_m'}{r\alpha_k}}v_k(Q_k)\Big)^{r\alpha_k}
\Big(\sum_{Q_k\in\mc{E}_k}|Q_k|\Big)^{\theta_k-r\alpha_k}\\
&\lesssim \|M^{\mc{D}(Q_0),v_k}(f_kv_k^{-1})\|_{L^{p_k}(Q_0,v_k)}^{p_m'}|Q_0|^{\theta_k-r\alpha_k}.
\end{align*}
The remainder of the estimate remains the same, this time noting that 
\[
\theta_0+\theta_k-r\alpha_k+\sum_{\substack{j=1\\j\neq k}}^{m-1}\theta_j-\alpha_j=\tfrac{p_m'}{p_0}.
\]
The result follows.
\end{proof}

\section{Proof of Theorem~\ref{thm:E} in the general case}\label{TheoremAPart2}

\subsection{Linearization}

We first linearize our sparse form domination by estimating the operator norm of $T$ by that of sparse operators. Below,
\[
A^q_{\mc{S}}\vec{f}:=\Big(\sum_{Q\in\mc{S}}\Big(\prod_{j=1}^m\langle f_j\rangle_{1,Q}\Big)^q\ind_Q\Big)^{\frac{1}{q}}
\]
for $q\in(0,\infty]$ with the usual modification when $q=\infty$. 
\begin{proposition}\label{prop:mainlinearization}
Let $T$ be an $m$-sublinear operator satisfying sparse form domination, let $\vec{p}\in[1,\infty]^m$, and let $p \in [\tfrac{1}{m},\infty)$ satisfy $\tfrac{1}{p} = \sum_{j=1}^m \tfrac{1}{p_j}$. If $\vec{w}\in A_{\vec{p}}$, then $T$ is bounded from $L^{\vec{p}}_{\vec{w}}(\R^d)$ to $L^{p,\infty}_w(\R^d)$ with 
\[
\|T\|_{L^{\vec{p}}_{\vec{w}}(\R^d)\to L^{p,\infty}_w(\R^d)}\lesssim\tfrac{1}{1-\frac{\theta}{p}}[\vec{w}]_{\vec{p}}^{1-\theta}\sup_{\mc{S}}\|A^\theta_{\mc{S}}\|_{L^{\vec{p}}_{\vec{w}}(\R^d)\to L_w^{p,\infty}(\R^d)}^\theta
\]
for all $\theta\in(0,p)\cap(0,1]$, where the supremum is taken over all $\tfrac{1}{2\cdot 6^d}$-sparse collections $\mc{S}$ contained in some dyadic grid.
\end{proposition}
For $\theta=1$ and $p>1$, an application of Kolmogorov's lemma shows that
\[
\|T\|_{L^{\vec{p}}_{\vec{w}}(\R^d)\to L^{p,\infty}_w(\R^d)}\lesssim p'\sup_{\mc{S}\text{ sparse}}\|A_{\mc{S}}\|_{L^{\vec{p}}_{\vec{w}}(\R^d)\to L_w^{p,\infty}(\R^d)},
\]
so the novelty in Proposition~\ref{prop:mainlinearization} is in the cases $\theta<p\leq 1$. 
\begin{proof}[Proof of Proposition~\ref{prop:mainlinearization}]
Let $E\subseteq\R^d$ with $0<v(E)<\infty$, where $v:=w^p$. By the sparse form domination assumption, \eqref{eq:kolmogorov2}, and the $3^d$-lattice theorem, it suffices to show that for each dyadic grid $\mc{D}$, there exists $E'\subseteq E$ with $v(E')\geq (1-\tfrac{1}{2\cdot 3^d})v(E)$ such that 
\[
\sum_{Q\in\mc{S}}\Big(\prod_{j=1}^m\langle f_j\rangle_{1,Q}\Big)\langle v\ind_E\rangle_{1,Q}|Q|\lesssim\tfrac{1}{1-\frac{\theta}{p}}[\vec{w}]_{\vec{p}}^{1-\theta}\sup_{\mc{S}\text{ sparse}}\|A^\theta_{\mc{S}}\|_{L^{\vec{p}}_{\vec{w}}(\R^d)\to L_w^{p,\infty}(\R^d)}^\theta,
\]
where the supremum is taken over all $\tfrac{1}{2\cdot 6^d}$-sparse collections $\mc{S}\subseteq\mc{D}$. Define
\[
\gamma:=\Big(\frac{2\cdot 3^d}{v(E)}\Big)^{\frac{1}{p}}[\vec{w}]_{\vec{p}},\quad \Omega:=\{x\in E:M^\mc{D}\vec{f}(x)>\gamma\},\quad \text{and}\quad E':=E\setminus\Omega.
\]
Since $\|M^\mc{D}\|_{L^{\vec{p}}_{\vec{w}}(\R^d)\to L^{p,\infty}_w(\R^d)}=[\vec{w}]_{\vec{p}}$, we have $v(\Omega)\leq\big(\tfrac{[\vec{w}]_{\vec{p}}}{\gamma}\big)^p=\tfrac{v(E)}{2\cdot 3^d}$ and so
\[
v(E')\geq v(E)-v(\Omega)\geq\big(1-\tfrac{1}{2\cdot 3^d}\big)v(E).
\]

Define
\[
\mc{S}_+:=\big\{Q\in\mc{S}:\prod_{j=1}^m\langle f_j\rangle_{1,Q}\leq\gamma\big\}.
\]
For any $Q\in\mc{S}\setminus\mc{S}_+$, we have $Q\subseteq\Omega$ so that $\langle v\ind_{E'}\rangle_{1,Q}=0$. Hence, we only need to consider the sum over $\mc{S}_+$. By Kolmogorov's lemma \eqref{eq:kolmogorov}, we have
\begin{align*}
\sum_{Q\in\mc{S}_+}\Big(\prod_{j=1}^m\langle f_j\rangle_{1,Q}\Big)&\langle v\ind_E\rangle_{1,Q}|Q|
\leq\gamma^{1-\theta}\int_E\!(A^\theta_{\mc{S}_+}\vec{f})^{\theta}v\,\mathrm{d}x\\
&\leq \tfrac{1}{1-\frac{\theta}{p}}\gamma^{1-\theta}\|A^\theta_{\mc{S}_+}\vec{f}\|_{L^{p,\infty}(\R^d,v)}^{\theta} v(E)^{1-\frac{\theta}{p}}.
\end{align*}
Since
\[
\gamma^{1-\theta} v(E)^{1-\frac{\theta}{p}}\eqsim[\vec{w}]_{\vec{p}}^{1-\theta}v(E)^{1-\frac{1}{p}},
\]
the assertion follows.
\end{proof}

\subsection{Good-\texorpdfstring{$\lambda$}{lambda} inequality}
We will need a version of the good-$\lambda$ technique from \cite[Theorem~E]{DFPR23}. For a collection of cubes $\mc{F}$, $a:=\{a_Q\}_{Q\in\mc{F}} \subseteq (0,\infty)$, and $r\in(0,\infty]$, we set
\[
A^r_{\mc{F}}(a):=\|\{a_Q\ind_Q\}_{Q\in\mc{F}}\|_{\ell^r(\mc{F})}. 
\]

\begin{theorem}\label{thm:goodlambdav}
If $\mc{D}$ is a dyadic grid, $w\in A_{\text{FW}}$, $\mc{S}\subseteq\mc{D}$ is a finite $\eta$-sparse collection, $a=\{a_Q\}_{Q\in\mc{S}}\subseteq (0,\infty)$, and $q,r\in(0,\infty]$ with $q<r$, then there exists $\delta>0$ such that  
\[
w(\{A^q_{\mc{S}}(a)>2\lambda,\, A^r_{\mc{S}}(a)\leq\gamma^{\frac{1}{q}-\frac{1}{r}}\lambda\})\lesssim e^{-\frac{\delta\eta}{\gamma[w]_{\text{FW}}}} w(\{A^q_{\mc{S}}(a)>\lambda\})
\]
for all $\lambda, \gamma>0$, where $\delta$ only depends on $d$, $q$, and $r$. 
\end{theorem}
\noindent As a consequence, we have that for all $p\in(0,\infty)$, $s\in(0,\infty]$, and $q\leq r$, we have
\[
\|A^q_{\mc{S}}(a)\|_{L^{p,s}(\R^d,w)}\lesssim_{s}\big(\tfrac{1}{\eta}[w]_{\text{FW}}\big)^{\frac{1}{q}-\frac{1}{r}}\|A^r_{\mc{S}}(a)\|_{L^{p,s}(\R^d,w)},
\]
where $\|f\|_{L^{p,s}(\R^d,w)}:= \big(\int_0^{\infty}(tw(\{|f|>t\}))^{\frac{s}{p}}\,\frac{dt}{t}\big)^{\frac{1}{s}}$ for $s<\infty$.

The proof of Theorem~\ref{thm:goodlambdav} relies on the following weighted John-Nirenberg inequality. 
\begin{lemma}\label{lem:jnheight}
If $\mc{D}$ is a dyadic grid, $w\in A_{\text{FW}}(\mc{D})$, $\mc{S}\subseteq\mc{D}$ is an $\eta$-sparse collection, $Q_0\in\mc{S}$, and 
$
h_{\mc{S}(Q_0)}:=\sum_{\substack{Q\in\mc{S}\\ Q\subseteq Q_0}}\ind_Q,
$
then there exists $\delta>0$ such that
\[
w\big(\{x\in Q_0:h_{\mc{S}(Q_0)}(x)>\lambda\}\big)\lesssim e^{-\frac{\eta\delta}{[w]_{\text{FW}}}\lambda}w(Q_0)
\]
for all $\lambda>0$, where $\delta$ only depends on $d$.
\end{lemma}
\begin{proof}
As $h_{\mc{S}(Q_0)}\in\text{BMO}(\mc{D})$, it follows from the John-Nirenberg inequality that
\[
|\{x\in Q_0:h_{\mc{S}(Q_0)}(x)>\lambda\}|\lesssim e^{-\eta\delta\lambda}|Q_0|.
\]
The result then follows from the sharp $A_{\text{FW}}$ condition
\[
\frac{w(E)}{w(Q)}\leq2\left(\frac{|E|}{|Q|}\right)^{\frac{1}{2^{d+1}[w]_{\text{FW}}}},
\]
which follows from 
Theorem \ref{thm:sharprh} with $Q=Q_0$ and $E=\{x\in Q_0:h_{\mc{S}(Q_0)}(x)>\lambda\}$.
\end{proof}
\begin{lemma}\label{lem:maximalcubes}
If $\mc{D}$ is a dyadic grid, $\mc{F}\subseteq\mc{D}$ is a finite collection of cubes, $a=\{a_Q\}_{Q\in\mc{F}} \subseteq (0,\infty)$, and $r\in(0,\infty]$, then for each $\lambda>0$, there exists a pairwise disjoint collection $\mc{Q}\subseteq\mc{D}$ such that $\{A^r_{\mc{F}}(a)>\lambda\}=\bigcup_{Q\in\mc{Q}}Q$ and the dyadic parent $\widehat{Q}$ of each $Q\in\mc{Q}$ intersects $\{A_{\mc{F}}^r(a)\leq\lambda\}$.
\end{lemma}
\begin{proof}
It suffices to show that for every $x\in E:=\{A^r_{\mc{F}}(a)>\lambda\}$, there is a cube $Q\in\mc{D}$ such that $x\in Q$ and $Q\subseteq E$. The collection of maximal cubes $\mc{Q}\subseteq\mc{D}$ contained in $E$ satisfies the desired properties.

Let $x\in E$ and let $Q(x):=\bigcap_{\substack{Q\in\mc{F}\\ x\in Q}}Q$. Then, as $\mc{F}$ is finite, $Q(x)\in\mc{D}$ by the intersection property of the dyadic grid. We claim that $Q(x)\subseteq E$. Indeed, let $y\in Q(x)$ and let $Q\in\mc{F}$ be a cube satisfying $x\in Q$. Then by definition of $Q(x)$, we also have $y\in Q$. Hence,
\[
\lambda<A^r_{\mc{F}}(a)(x)=\Big(\sum_{\substack{Q\in\mc{F}\\ Q\ni x}}a_Q^r\Big)^{\frac{1}{r}}\leq \Big(\sum_{\substack{Q\in\mc{F}\\ Q\ni y}}a_Q^r\Big)^{\frac{1}{r}}=A^r_{\mc{F}}(a)(y).
\]
We conclude that $y\in E$, proving the claim. The result follows.
\end{proof}

\begin{proof}[Proof of Theorem~\ref{thm:goodlambdav}]
By homogeneity, it suffices to prove the case $\lambda=1$. Set
\[
\Omega:=\{A^q_{\mc{S}}(a)>2,\, A^r_{\mc{S}}(a)\leq\gamma^{\frac{1}{q}-\frac{1}{r}}\}.
\]
Use Lemma~\ref{lem:maximalcubes} to decompose $\{A^q_{\mc{S}}(a)>1\}=\bigcup_{Q\in\mc{Q}}Q$, where $\mc{Q}\subseteq\mc{D}$ is the disjoint collection of maximal cubes in $\{A^q_{\mc{S}}(a)>1\}$. As the $Q\in \mc{Q}$ cover $\Omega$, it suffices to prove 
\[
w(\Omega\cap Q)\lesssim e^{-\frac{\delta\eta}{\gamma[w]_{\text{FW}}}} w(Q)
\]
for all $Q\in\mc{Q}$.
Fix $Q\in\mc{Q}$ and pick $\widehat{x}\in\widehat{Q}$ for which $A^q_{\mc{S}}(a)(\widehat{x})\leq 1$. Then we have 
\[
2^q<A^q_{\mc{S}(Q)}(a)(x)^q+\sum_{\substack{Q'\in\mc{S}\\\widehat{Q}\subseteq Q'}}a_{Q'}^q
\leq A_{\mc{S}(Q)}(a)(x)^q+A^q_{\mc{S}}(a)(\widehat{x})^q\leq A^q_{\mc{S}(Q)}(a)(x)^q+1
\]
for $x\in\Omega\cap Q$. Hence, by H\"older's inequality, we have
\[
(2^q-1)^{\frac{1}{q}}<A^q_{\mc{S}(Q)}(a)(x)\leq A^r_{\mc{S}(Q)}(a)(x)h_{\mc{S}(Q)}(x)^{\frac{1}{q}-\frac{1}{r}}\leq(\gamma h_{\mc{S}(Q)}(x))^{\frac{1}{q}-\frac{1}{r}}.
\]
By Lemma~\ref{lem:jnheight}, there exists $\delta>0$ depending on $d$, $q$, and $r$ such that
\[
w(\Omega\cap Q)\leq w\Big(\Big\{x\in Q: h_{\mc{S}(Q)}(x)>(2^q-1)^{\frac{\frac{1}{q}}{\frac{1}{q}-\frac{1}{r}}}\tfrac{1}{\gamma}\Big\}\Big)\lesssim e^{-\frac{\delta\eta}{\gamma[w]_{\text{FW}}}} w(Q),
\]
proving the first assertion. The second follows from a standard good-$\lambda$ argument.
\end{proof}

\subsection{Proof of Theorem~\ref{thm:E} in the general case}
\begin{proof}[Proof of Theorem~\ref{thm:E} in the general case]
Set $v:=w^p$. Let $\mc{D}$ be a dyadic grid, $\mc{S}\subseteq\mc{D}$ be an $\tfrac{1}{2\cdot 6^d}$-sparse collection, and $\vec{f}\in L^{\vec{1}}_{\vec{w}}(\R^d)$. By monotone convergence, we may assume that $\mc{S}$ is finite. By Theorem~\ref{thm:goodlambdav} with $a_Q=\prod_{j=1}^m\langle f_j\rangle_{1,Q}$, $q=\theta=\tfrac{1}{2m}$, and $r=\infty$, we have
\begin{align*}
\|A^\theta_{\mc{S}}\vec{f}\|_{L^{p,\infty}(\R^d,v)}
&\lesssim ([v]_{\text{FW}})^{\frac{1}{\theta}}\|A^\infty_{\mc{S}}(a)\|_{L^{p,\infty}(\R^d,v)}\\
&\leq([v]_{\text{FW}})^{\frac{1}{\theta}}\|M^{\mc{D}}\vec{f}\|_{L_w^{p,\infty}(\R^d)}\\
&\leq([v]_{\text{FW}})^{\frac{1}{\theta}}[\vec{w}]_{\vec{p}}\|\vec{f}\|_{L_{\vec{w}}^{\vec{p}}(\R^d)}.
\end{align*}
Thus, the result follows from Proposition~\ref{prop:mainlinearization}.
\end{proof}

\section{Proof of Theorem \ref{MLMultiplierWeakType}}\label{TheoremBSection}

We prove the cases $p\geq 1$ and $p<1$ separately. We start with the case $p\geq 1$.

\begin{proof}[Proof of Theorem~\ref{MLMultiplierWeakType} in the case $p\geq 1$]
    As in the proof of Proposition~\ref{prop:mainlinearization}, using \eqref{eq:kolmogorov}, the sparse form domination
    \[
    \int_{E'}\!|T(\vec{f}/\vec{w})| w\,\mathrm{d}x\lesssim\sum_{Q\in\mc{S}}\Big(\prod_{j=1}^m\langle f_jw_j^{-1}\rangle_{1,Q}\Big)\langle w\ind_{E'}\rangle_{1,Q}|Q|,
    \]
    and the $3^d$ lattice theorem, it suffices to show that for every $E\subset\R^d$ with $0<|E|<\infty$ and every dyadic grid $\mc{D}$, there exists a set $E'\subseteq E$ with $|E'|\geq (1-\tfrac{1}{2\cdot 3^d})|E|$ such that for all $\tfrac{1}{2\cdot 6^d}$-sparse collections $\mc{S}\subset\mc{D}$, we have
    \begin{equation}\label{eq:MLMultiplierWeakTypeeq1}
    \sum_{Q\in\mc{S}}\Big(\prod_{j=1}^m\langle f_jw_j^{-1}\rangle_{1,Q}\Big)\langle w\ind_{E'}\rangle_{1,Q}|Q|\lesssim [w^p]_{\text{FW}}[\vec{w}]_{\vec{p}}|E|^{1-\frac{1}{p}}
    \end{equation}
    for all non-negative $f_j\in L^{p_j}(\R^d)$ with $\|f_j\|_{L^{p_j}(\R^d)}=1$ and $j\in\{1,\ldots,m\}$. 
    
    Let $E \subseteq \R^d$ with $0 < |E| < \infty$ and let $\mc{D}$ be a dyadic grid. For a positive constant $K$ to be fixed below, we define for each $j \in \{1, \ldots m\}$ such that $p_j \neq \infty$
    \[
        \Omega_j := \{x \in \R^d : M^\mathcal{D}(f_j^{p_j})(x) > \tfrac{K}{|E|}\}.
      \]
    Forming the Calderón-Zygmund decomposition of $f_j^{p_j}$ 
    at height $\tfrac{K}{|E|}$, we obtain a collection of disjoint cubes $\mc{P}_j\subseteq\mc{D}$ and functions $g_j$ and $b_j$ such that 
    \begin{align*}
        \Omega_j = \bigcup_{P\in\mc{P}_j} P, &
        \quad f_j^{p_j} = g_j + b_j,
        \quad \|g_j\|_{L^1(\R^d)} \lesssim 1,
        \quad \|g_j\|_{L^\infty(\R^d)} \lesssim K/|E|,\\
        & \text{supp}(b_j) \subseteq \Omega_j,
        \quad \text{and}\quad \avg{b_j}_{P} = 0 \text{ for all $P\in\mc{P}_j$}.
    \end{align*}
    Since $\|M^\mathcal{D}\|_{L^1(\R^d)\rightarrow L^{1,\infty}(\R^d)}=1$ and $\|f_j\|_{L^{p_j}(\R^d)} = 1$, fixing $K= 2m\cdot 3^d$ we have
    \[
        |\Omega_j| = \big|\big\{x \in \R^d : M^D(f_j^{p_j})(x) > \tfrac{K}{|E|}\big\}\big| \leq \tfrac{|E|}{K}=\tfrac{|E|}{2m\cdot 3^d}.
    \]
    Setting $\Omega := \bigcup_{\{j : p_j \neq \infty\}} \Omega_j$ and $E' := E \setminus \Omega$, we have $|E'| \geq(1-\tfrac{1}{2\cdot 3^d})|E|$.
    
    Since $w^p \in A_{\text{FW}}$, for $\nu\in(1,\infty)$ defined through $\nu' = 2^{d+1}[w^p]_{\text{FW}}$, it follows from the sharp reverse Hölder inequality Theorem~\ref{thm:sharprh} that
    \[
    \langle w\rangle_{p\nu,Q}=\langle w^p\rangle_{\nu,Q}^{\frac{1}{p}}\lesssim_p \langle w^p\rangle_{1,Q}^{\frac{1}{p}}=\langle w\rangle_{p,Q}
    \]
    for all $Q\in\mc{D}$. Fix $r$ such that $r' = (p\nu)' + 1$. Then we have
    \[
        1 < r < \nu, \qquad (r')^r \lesssim \nu' \lesssim [w]_{\text{FW}},\qquad \text{and}\qquad\frac{(pr)'}{(p\nu)'} =  \frac{1}{p} + \frac{1}{(p\nu)'} = r.
    \]    
    By our assumption that $\|f_j\|_{L^{p_j}(\R^d)} = 1$ for $j \in \{1, \ldots, m\}$, we have for $p_j=\infty$ that 
    \[
        \avg{f_jw_j^{-1}}_{1,Q} \leq \|f_j\|_{L^\infty(\R^d)}\avg{w_j^{-1}}_{1,Q} = \avg{w_j^{-1}}_{p_j',Q}.
    \]
    Applying the sparse bound for $T$, Hölder's inequality for $p_j \neq \infty$, the $A_{\vec{p}}$ condition, the sharp reverse H\"older inequality for $w^p$, and the sparseness of the collection $\mathcal{S}$, we have
    \begin{align*}
        \sum_{Q\in\mathcal{S}}\Big(\prod_{j=1}^m&\avg{f_jw_j^{-1}}_{1,Q}\Big)\avg{ w\ind_{E'}}_{1,Q}|Q| \\
        &\lesssim \sum_{Q \in \mathcal{S}}\Big(\prod_{p_j \neq \infty}\avg{f_j}_{p_j,Q}\Big)\Big( \prod_{j=1}^{m}\avg{w_j^{-1}}_{p_j',Q}\Big)\avg{w}_{p\nu,Q}\avg{\ind_{E'}}_{(p\nu)',Q}|Q| \\
        &\lesssim \sum_{Q \in \mathcal{S}}\Big(\prod_{p_j \neq \infty}\avg{f_j}_{p_j,Q}\Big)\Big(\prod_{j=1}^{m}\avg{w_j^{-1}}_{p_j',Q}\Big)\avg{w}_{p,Q}\avg{\ind_{E'}}_{(p\nu)',Q}|Q|  \\
        &\lesssim [\vec{w}]_{A_{\vec{p}}} \sum_{Q \in \mathcal{S}}\Big(\prod_{p_j \neq \infty}\avg{f_j}_{p_j,Q}\Big)\avg{\ind_{E'}}_{(p\nu)',Q}|E_Q|.
    \end{align*}
    
    Let $Q \in \mathcal{S}$. If $Q \subseteq \Omega$, then $\avg{\ind_{E'}}_{(p\nu)',Q} = 0$, since $E' \cap \Omega = \emptyset$. Therefore, the non-zero terms in the above sum correspond to $Q$ that intersect $\R^d \setminus \Omega$. For such $Q$, if $Q \cap P \neq \emptyset$ then either $Q \subseteq P$ or $P \subseteq Q$, and since $P \subseteq \Omega$, we must have that $P \subseteq Q$.  Therefore, for each $j \in \{1,\ldots, m\}$ such that $p_j \neq \infty$, we have
    \[
      \avg{f_j}_{p_j,Q} = (\avg{g_j}_Q + \avg{b_j}_Q)^{\frac{1}{p_j}} = \Big(\avg{g_j}_Q + |Q|^{-1}\sum_{\substack{P\in\mc{P}_j\\ P \subseteq Q}} \int_P\!b_j\,\mathrm{d}x\Big)^{\frac{1}{p_j}} = \avg{g_j}_Q^{\frac{1}{p_j}},
    \]
    since $\avg{b_j}_{P} = 0$ for any $P\in\mc{P}_j$. We estimate the final term above with H\"older's inequality with exponents $pr$ and $p_jr$, the norm bounds for $g_j$, and the boundedness of $M$:
    \begin{align*}
        \sum_{Q \in \mathcal{S}}\Big(\prod_{p_j \neq \infty}\avg{f_j}_{p_j,Q}\Big)&\avg{\ind_{E'}}_{(p\nu)',Q}|E_Q|= \sum_{Q \in \mathcal{S}}\Big(\prod_{p_j \neq \infty}\avg{g_j}_{1,Q}^\frac{1}{p_j}\Big)\avg{\ind_{E'}}_{(p\nu)',Q}|E_Q|\\
        &\leq \sum_{Q \in \mathcal{S}}\int_{E_Q}\!\Big(\prod_{p_j \neq \infty}(Mg_j)^{\frac{1}{p_j}}\Big)M(\ind_{E'})^{\frac{1}{(p\nu)'}}\,dx\\
        &\leq \Big(\prod_{p_j \neq \infty}\|Mg_j\|_{L^{r}(\R^d)}^{\frac{1}{p_j}}\Big)\|M(\ind_{E'})\|_{L^r(\R^d)}^{\frac{1}{(p\nu)'}}\\
        &\lesssim (r')^{\frac{1}{p}}(r')^{\frac{1}{(p\nu)'}}\Big(\prod_{p_j \neq \infty}\|g_j\|_{L^r(\R^d)}^{\frac{1}{p_j}}\Big)|E'|^{\frac{1}{(pr)'}}\\ &\lesssim[w^p]_{\text{FW}}\Big(\prod_{p_j \neq \infty}(\|g_j\|_{L^\infty(\R^d)}^{\frac{1}{r}}\|g_j\|_{L^1(\R^d)}^{\frac{1}{r}})^{\frac{1}{p_j}}\Big)|E'|^{\frac{1}{(pr)'}}\\
        &\lesssim [w^p]_{\text{FW}}|E|^{\frac{1}{(pr)'}-\frac{1}{pr'}}\\
        &= [w^p]_{\text{FW}}|E|^{1-\frac{1}{p}}.
    \end{align*}
    Combining these two estimates yields \eqref{eq:MLMultiplierWeakTypeeq1},
    as desired.
\end{proof}

\begin{proof}[Proof of Theorem~\ref{MLMultiplierWeakType} in the case $p<1$]
    Observe that
        \[
            \|T(\vec{f}/\vec{w})w\|_{L^{p,\infty}(\R^n)} = \|T(\vec{f}/\vec{w})^pw^p\|_{L^{1,\infty}(\R^d)}^{\frac{1}{p}
            }.
        \]
    We proceed as in the proof of the case $p\geq 1$, replacing the sparse form domination with sparse form domination of $\ell^p$ type, to see that
    \[
    \int_{E'}\!|T(\vec{f}/\vec{w})|^p w\,\mathrm{d}x\lesssim\sum_{Q\in\mc{S}}\Big(\prod_{j=1}^m\langle f_jw_j^{-1}\rangle_{1,Q}\Big)^p\langle w^p\ind_{E'}\rangle_{1,Q}|Q|.
    \]
    We next show that for every $E\subseteq\R^d$ with $0<|E|<\infty$ and every dyadic grid $\mc{D}$ there exists a set $E'\subseteq E$ with $|E'|\geq (1-\tfrac{1}{2\cdot 3^d})|E|$ so that for all $\tfrac{1}{2\cdot 6^d}$-sparse collections $\mc{S}\subseteq\mc{D}$, we have
    \begin{equation}\label{eq:MLMultiplierWeakTypeeq2}
    \sum_{Q\in\mc{S}}\Big(\prod_{j=1}^m\langle f_jw_j^{-1}\rangle_{1,Q}\Big)^p\langle w^p\ind_{E'}\rangle_{1,Q}|Q|\lesssim [w^p]^p_{\text{FW}}[\vec{w}]^p_{\vec{p}}
    \end{equation}
    for all non-negative $f_j\in L^{p_j}(\R^d)$ with $\|f_j\|_{L^{p_j}(\R^d)}=1$. 
    
    Let $E\subseteq\R^d$ with $0<|E|<\infty$, let $\mc{D}$ be a dyadic grid, and define $E'$ and $\nu\in(1,\infty)$ exactly as in the proof of the case $p\geq 1$. For $r\in(1,\infty)$ such that $r' = 2\nu'$, we have
        \[
            1 < r < \nu\quad\text{and}\quad r' \eqsim [w^p]_{\text{FW}}.
        \]    
        Similar to the argument for $p \geq 1$, by Hölder's inequality for the $p_j \neq \infty$, the $A_{\vec{p}}$ condition, the sharp reverse H\"older inequality for $w^p$, and the sparseness of the collection $\mathcal{S}$, we have
        \begin{align*}
         \sum_{Q\in\mathcal{S}}\Big(\prod_{j=1}^{m} &\avg{f_jw_j^{-1}}_{1,Q}\Big)^p\avg{ w^p\ind_{E'}}_{1,Q}|Q| \\
          &\lesssim \sum_{Q \in \mathcal{S}}\Big(\prod_{p_j \neq \infty}\avg{f_j}_{p_j,Q}\Big)^p\Big( \prod_{j=1}^{m}\avg{w_j^{-1}}_{p_j',Q}\Big)^p\avg{w^p}_{\nu,Q}\avg{\ind_{E'}}_{\nu',Q}|Q| \\
          &\lesssim \sum_{Q \in \mathcal{S}}\Big(\prod_{p_j \neq \infty}\avg{f_j}_{p_j,Q}\Big)^p\Big( \prod_{j=1}^{m}\avg{w_j^{-1}}_{p_j',Q}\Big)^p\avg{w^p}_{1,Q}\avg{\ind_{E'}}_{\nu',Q}|Q| \\
          &\lesssim [\vec{w}]_{\vec{p}}^p 
          \sum_{Q \in \mathcal{S}}\Big(\prod_{p_j \neq \infty}\avg{f_j}_{p_j,Q}\Big)^p\avg{\ind_{E'}}_{\nu',Q}|E_Q|.
        \end{align*}
        We estimate the final term above using H\"older's inequality with exponents $r$ and $p_jr$, the norm bounds for $g_j$ and the operator bounds for the maximal operator: 
        \begin{align*}
        \sum_{Q \in \mathcal{S}}\Big(\prod_{p_j \neq \infty}\avg{f_j}_{p_j,Q}\Big)^p&\avg{\ind_{E'}}_{\nu',Q}|E_Q|= \sum_{Q \in \mathcal{S}}\Big(\prod_{p_j \neq \infty}\avg{g_j}_{1,Q}^{\frac{p}{p_j}}\Big)\avg{\ind_{E'}}_{\nu',Q}|E_Q|\\
            &\leq \sum_{Q \in \mathcal{S}}\int_{E_Q}\Big(\prod_{p_j \neq \infty}(Mg_j)^{\frac{p}{p_j}}\Big)M(\ind_{E'})^{\frac{1}{\nu'}}\,dx\\
            &\leq \Big(\prod_{p_j \neq \infty}\|Mg_j\|_{L^{r}(\R^d)}^{\frac{p}{p_j}}\Big)\|M(\ind_{E'})\|_{L^2(\R^d)}^{\frac{1}{r'}}\\
            &\lesssim (r')^p\Big(\prod_{p_j \neq \infty}\|g_j\|_{L^r(\R^d)}^{\frac{p}{p_j}}\Big)|E'|^{\frac{1}{r'}}\\
            &\lesssim [w^p]_{\text{FW}}^p|E|^{\frac{1}{r'}-\frac{1}{r'}}\\
            &= [w^p]_{\text{FW}}^p.
        \end{align*}
    Combining these two estimates implies \eqref{eq:MLMultiplierWeakTypeeq2}, as desired.
\end{proof}

\appendix

\section{Concluding remarks} 
\label{app:A}



It would be interesting to obtain a version of Theorem~\ref{thm:E} that unifies \eqref{eq:zor1} and Theorem~\ref{thm:E} 
in terms of the quantity
\[
[\vec{w}]_{\text{FW}_{\text{prod}}}^{\vec{p}}:=\sup_Q\Big(\prod_{j=1}^m v_j(Q)^{\frac{1}{p_j}}\Big)^{-1}\Big(\int_Q\!M_{\vec{p}}\big(v_1^{\frac{1}{p_1}}\ind_Q,\ldots,v_m^{\frac{1}{p_m}}\ind_Q\big)^p\,\mathrm{d}x\Big)^{\frac{1}{p}},
\]
where $v_j:= w_j^{-p_j'}$. 
We here prove 
\[
[\vec{w}]_{\text{FW}_{\text{prod}}}^{\vec{p}}\lesssim\min_{j\in\{1,\ldots,m\}}[v_j]_{\text{FW}}^{\frac{1}{p}}\leq[\vec{w}]_{\vec{p}}^{\min\big(\frac{p_1'}{p},\ldots,\frac{p_m'}{p}\big)}
\]
for $\vec{w}\in A_{\vec{p}}$. We refer to 
\cite[Section~3.3]{Ni20} for further discussion of these constants.

\begin{proposition}\label{AppendixProp}
If $\vec{p}\in(1,\infty]^m$, $p \in (\tfrac{1}{m},\infty)$ satisfies $\frac{1}{p}=\sum_{j=1}^{m} \frac{1}{p_j}$, and $\vec{w}$ are weights such that $v_j\in A_{\text{FW}}$ for some $j \in \{1,\ldots,m\}$, then
\[
[\vec{w}]_{\text{FW}_{\text{prod}}}^{\vec{p}}\lesssim\min_{j\in\{1,\ldots,m\}}[v_j]_{\text{FW}}^{\frac{1}{p}},
\]
where $v_j:=w_j^{-p_j'}$. Moreover, if $\vec{w}\in A_{\vec{p}}$, then 
\[
\min_{j\in\{1,\ldots,m\}}[v_j]_{\text{FW}}^{\frac{1}{p}}\leq[\vec{w}]_{\vec{p}}^{\min\big(\frac{p_1'}{p},\ldots,\frac{p_m'}{p}\big)}.
\]
\end{proposition}
\begin{proof}
It is shown in \cite[Remark~3.3.2]{Ni20} that $[\vec{v}]_{\text{FW}_{\text{prod}}}^{\vec{p}}<\infty$ if and only if for all sparse collections $\mc{S}$ in a dyadic grid $\mc{D}$ and all $Q\in\mc{D}$ we have
\[
\sum_{Q'\in\mc{S}(Q)}\Big(\prod_{j=1}^m\langle v_j\rangle_{1,Q'}^{\frac{p}{p_j}}\Big)|Q'|\lesssim \Big(\prod_{j=1}^m\langle v_j\rangle_{1,Q}^{\frac{p}{p_j}}\Big)|Q|,
\]
and that the optimal constant is equivalent to $\big([\vec{v}]_{\text{FW}_{\text{prod}}}^{\vec{p}}\big)^p$. Without loss of generality, assume that $v_m\in A_{\text{FW}}$. Define $r\in(1,\infty)$ through $r'=2^{d+1}[v_m]_{\text{FW}}$. It follows from the sharp reverse H\"older inequality from Theorem \ref{thm:sharprh} that
\begin{equation}\label{eq:appa1}
\langle v_m^r\rangle_{1,Q}^{\frac{1}{r}}\leq 2\langle v_m\rangle_{1,Q}
\end{equation}
for all $Q\in\mc{D}$. Set $\alpha_j:=\tfrac{p}{p_j}$ for $j\in\{1,\ldots,m-1\}$ and $\alpha_m:=\tfrac{1}{r}\tfrac{p}{p_m}$. Then
\[
\sum_{j=1}^m\alpha_j=1-\frac{1}{r'}\frac{p}{p_m}<1.
\]
Defining
\[
\theta_j:=\alpha_j+\frac{1}{m}\Big(1-\sum_{j=1}^m\alpha_j\Big),
\]
we have $\theta_j=\tfrac{p}{p_j}+\tfrac{1}{m}\tfrac{1}{r'}\tfrac{p}{p_m'}$ for $j\in\{1,\ldots,m-1\}$ and $\theta_m=\tfrac{1}{r}\tfrac{p}{p_j}+\tfrac{1}{m}\tfrac{1}{r'}\tfrac{p}{p_m'}$. Exactly as in the proof of Lemma~\ref{lem:multilineartesting3}, we then find that
\begin{align*}
\sum_{Q'\in\mc{S}(Q)}\Big(\prod_{j=1}^m\langle v_j\rangle_{1,Q'}^{\frac{p}{p_j}}\Big)|Q'|&\leq \sum_{Q'\in\mc{S}(Q)}\Big(\prod_{j=1}^{m-1}\langle v_j\rangle_{1,Q'}^{\frac{p}{p_j}}\Big)\langle v_m^r\rangle_{1,Q'}^{\frac{1}{r}\frac{p}{p_m'}}|Q'|\\
&\lesssim\Big(\prod_{j=1}^m\big(\tfrac{1}{1-\frac{\alpha_j}{\theta_j}}\big)^{\theta_j}\Big)\Big(\prod_{j=1}^{m-1}\langle v_j\rangle_{1,Q}^{\frac{p}{p_j}}\Big)\langle v_m^r\rangle_{1,Q}^{\frac{1}{r}\frac{p}{p_m'}}|Q|\\
&\lesssim (r')^{\sum_{j=1}^m\theta_j}\prod_{j=1}^m\langle v_j\rangle_{1,Q}^{\frac{p}{p_j}}\\
&\eqsim [v_m]_{\text{FW}}\prod_{j=1}^m\langle v_j\rangle_{1,Q}^{\frac{p}{p_j}},
\end{align*}
where in the last inequality we used \eqref{eq:appa1}. This proves the first result.

The second property holds since
\[
[v_j]_{\text{FW}}\lesssim[v_j]_{A_{mp_j'}}\leq[\vec{w}]^{p_j'}_{\vec{p}}.
\]
The result follows.
\end{proof}


Proposition \ref{AppendixProp} and the bound
$[\vec{w}]_{\text{FW}_{\text{prod}}}^{\vec{p}}\leq[\vec{v}]_{\text{FW}}^{\vec{p}}\lesssim[\vec{w}]_{\vec{p}}^{\max\big(\tfrac{p_1'}{p_1},\ldots,\tfrac{p_m'}{p_m}\big)}$ from \cite{Zor19} give
\[
[\vec{w}]_{\text{FW}_{\text{prod}}}^{\vec{p}}\lesssim [\vec{w}]_{\vec{p}}^{\min(\gamma,\delta)},
\]
where
\[
\gamma:=\min\Big(\frac{p_1'}{p},\ldots,\frac{p_m'}{p}\Big) \quad\text{and}\quad
\delta:=\max\Big(\frac{p_1'}{p_1},\ldots,\frac{p_m'}{p_m}\Big).
\]
While $\tfrac{p_j'}{p}\geq\tfrac{p_j'}{p_j}$ for all $j\in\{1,\ldots,m\}$, whether $\gamma$ or $\delta$ is smaller depends on $\vec{p}$.  

We conclude by remarking that our proofs extend to the two-weight setting where the product weight $w$ is replaced by a general weight. Moreover, the bound 
\[\|T\vec{f}\|_{L^{p,\infty}_w(\R^d)}\lesssim [\vec{w}]^{\min(\alpha,\beta)}_{\vec{p}}\|\vec{f}\|_{L^{\vec{p}}_{\vec{w}}(\R^d)}
\] 
of Theorem \ref{thm:E} can be extrapolated beyond the restrictions $p>1$ and $p_j<\infty$ with \cite[Theorem~4.10]{Ni19}, and such an extrapolation yields a smaller exponent than $\beta$ in this extended range. We leave these details to the interested reader.





\bibliography{babel}
\bibliographystyle{alpha-sort}
\end{document}